\documentclass[12pt,a4paper]{article}
\usepackage{fontspec}
\usepackage[english]{babel}
\usepackage{amsmath,amssymb,amsthm,mathtools}
\usepackage{dsfont}
\numberwithin{equation}{section}
\usepackage{geometry}
\usepackage{booktabs,array}
\usepackage{xcolor}
\usepackage{csquotes}
\usepackage{tikz}
\usetikzlibrary{arrows.meta,positioning,calc}
\usepackage[
 backend=biber,
 style=numeric,
 sorting=none,
 giveninits=true,
 maxbibnames=99
]{biblatex}
\DeclareNameAlias{author}{given-family}
\DeclareNameAlias{editor}{given-family}
\DeclareFieldFormat[article]{title}{#1}
\DeclareFieldFormat[article]{number}{\mkbibparens{#1}}
\renewbibmacro*{volume+number+eid}{%
 \printfield{volume}%
 \setunit*{\addspace}%
 \printfield{number}%
 \setunit{\bibeidpunct}%
 \printfield{eid}}
\renewbibmacro{in:}{}
\defbibheading{bibliography}[\refname]{\section*{\centering #1}}
\AtBeginBibliography{\sloppy}
\usepackage{hyperref}
\hypersetup{
 pdftitle={The Zaporozhets--Tarasov Inequality for an Arbitrary Planar Convex Body},
 pdfsubject={A complete proof of the strict mean-distance inequality},
 pdfkeywords={convex body, mean distance, boundary, Gini mean difference},
 pdfauthor={Maksim Kukushkin}
}
\newif\ifprintversion
\ifprintversion
 \hypersetup{hidelinks}
\else
 \hypersetup{colorlinks=true,linkcolor=blue!55!black,urlcolor=blue!55!black}
\fi
\usepackage{microtype}
\usepackage{enumitem}
\setlist{nosep}

\newtheorem{theorem}{Theorem}
\newtheorem{lemma}[theorem]{Lemma}
\newtheorem{proposition}[theorem]{Proposition}
\theoremstyle{definition}
\newtheorem{definition}{Definition}
\theoremstyle{remark}
\newtheorem{remark}{Remark}
\theoremstyle{plain}

\newcommand{\E}{\mathbb E}
\newcommand{\R}{\mathbb R}
\newcommand{\Var}{\operatorname{Var}}
\newcommand{\co}{\operatorname{co}}
\newcommand{\dd}{\,\mathrm{d}}

\let\leq\leqslant
\let\geq\geqslant
\newcommand*{\eq}{\mathrel{\overset{\mathrm{def}}{=}}}
\newcommand{\step}[1]{\par\medskip\noindent\textbf{#1}\par\smallskip}

\title{The Zaporozhets--Tarasov Inequality\\
for an Arbitrary Planar Convex Body}
\author{Maksim Kukushkin\\[-0.1em]
\small Faculty of Mathematics and Computer Science\\[-0.15em]
\small Saint Petersburg State University, Saint Petersburg, Russia}
\date{Russian original: 26 July 2026\\
English translation: 13 August 2026}

\begin{document}
\maketitle

\begin{abstract}
We prove that the mean distance between two independent uniformly distributed
points in an arbitrary planar convex body is strictly smaller than the mean
distance between two independent uniformly distributed points on its boundary.
Equality is impossible, although the difference between these two mean distances
tends to zero along a sequence of thin rectangles. The proof reduces the problem
to a one-dimensional comparison of two Gini mean differences. The main new
ingredient is a moment lemma for a random pair
$(\varepsilon,R)\in\{-1,1\}\times[0,1]$. The only finite algebraic part of the
proof is given by exact rational certificates in the Bernstein basis; the
verification script and all 6492 coefficients are available in the supplementary
materials.
\end{abstract}

\medskip
\begin{center}
\begin{minipage}{0.88\textwidth}\small
\noindent\textbf{Keywords:} convex body, mean distance,
boundary, Gini mean difference, Bernstein basis.
\end{minipage}
\end{center}
\section{Introduction}\label{sec:introduction}

The Zaporozhets--Tarasov conjecture compares mean distances for the uniform
distributions inside a planar convex body and on its boundary. The principal
difficulty is that these two measures generally have different centroids and
need not be centrally symmetric, so a direct pointwise comparison of distances
is unavailable.

According to A.~S.~Lotnikov~\cite{Lotnikov2025}, the conjecture was proposed
by D.~N.~Zaporozhets and A.~S.~Tarasov in 2019. The earliest published
formulation that we found appears in a paper by
A.~S.~Tokmachev~\cite{Tokmachev2022}. In both sources the conjecture is stated
with a strict inequality. We prove this strict form below for every planar
convex body. Equality is impossible, although for thin rectangles the
difference between the two mean distances tends to zero.

One of the initial approaches was based on normalizing mean distances by the
perimeter. For the uniform measure in the interior, Bonnet, Gusakova, Thäle,
and Zaporozhets~\cite{BonnetEtAl2021} proved, in particular, the sharp planar
bounds
$7/60<\E|I_1-I_2|/P(K)<1/6$. Therefore, for the initial hypothesis it would be sufficient
to establish the lower bound $\E|B_1-B_2|/P(K)>1/6$.

A.~S.~Tokmachev~\cite{Tokmachev2022} proved the complementary upper bound
$\E|B_1-B_2|/P(K)\leq 2/\pi^2$ using Fourier series and approximation of
convex bodies; this method did not yield the required lower bound.

A.~S.~Lotnikov~\cite{Lotnikov2025} proved the conjecture for centrally symmetric
planar bodies, simultaneously for all moments of the distance of order $p\geq1$.
His proof uses a stochastic comparison of one-dimensional projections; that
paper also gives results for sufficiently high moments in more general bodies.

Later A.~S.~Tokmachev~\cite{Tokmachev2026} established a stronger comparison
for separately convex functions of several points under the additional assumption
that the centroids of the interior and boundary uniform measures coincide. His
method is based on the convex order of one-dimensional projections.

In the present paper we assume neither central symmetry nor coincidence of
centroids, but consider only the first moment of the distance between two points.
The method does not require a stochastic comparison of individual projections:
the projections of area and perimeter are described by a pair of monotone
densities, reducing the problem to a comparison of Gini mean differences, a
\emph{moment lemma}, and finitely many rational \emph{Bernstein certificates}.

\subsection{Notations and conventions}

\begin{definition}[monotonicity convention]\label{def:monotonicity}
A function $w$ is called \emph{increasing} if $x<y$ implies
$w(x)\leq w(y)$, and \emph{decreasing} if $x<y$ implies $w(x)\geq w(y)$.
If the corresponding inequality is always strict, we use the terms
\enquote{strictly increasing} and \enquote{strictly decreasing}.
\end{definition}

Differentials are typeset in upright font: $\dd x=\,\mathrm{d}x$; the symbol
$\eq$ denotes equality by definition.

The principal notation and local reuse of symbols are summarized below.
\begin{center}
\small
\begin{tabular}{@{}p{31mm}p{108mm}@{}}
\toprule
notation & meaning\\
\midrule
$K,|K|,P(K)$ & a convex body, its area, and its perimeter\\
$I,B$ & uniform random points in $K$ and on $\partial K$, respectively\\
$D(\sigma)$ & Gini mean difference; the precise definition is given in
Section~\ref{sec:one-dimensional}\\
$u(t),v(t),h(t)$ & the upper and lower boundaries of the body and their difference;
in the appendix, $u$ locally denotes $\overline\E[\gamma(R)]$.
The symbol $\overline\E$ is introduced in~\eqref{eq:bar-expectation}\\
$q(t)$ & the unnormalized linear density of boundary length; in the appendix,
$q=a-e$ denotes the mean of a conditional distribution\\
$e_\theta$ & a unit projection direction; in the appendix, $e$ denotes
$\E(\varepsilon R)$\\
$\varphi_0,\varphi_1$ & orthonormal functions in Bessel's inequality\\
$\alpha,\beta$ & the Stieltjes measures $-\dd f$ and $\dd g$; in
Appendix~\ref{app:atomic-check}, $\alpha$ is a multi-index and
$\beta_k$ are Bernstein coefficients\\
$t$ & a projection parameter; in the appendix, the point of tangency of the right chord\\
$x,y$ & variables of integration or, locally in the appendix, endpoints of a chord\\
\bottomrule
\end{tabular}
\end{center}

\subsection{Structure of the paper}

The logical dependencies are shown in Figure~\ref{fig:proof-map}. The
geometric reduction is proved in Section~\ref{sec:geometry}, the
one-dimensional lemma in Section~\ref{sec:one-dimensional}, and the technical
proof of the moment lemma is deferred to Appendix~\ref{app:moment-proof}.
Appendix~\ref{app:supplement} proves standard facts about the countability of
maximal support segments, derives the projection densities of area and boundary
length, proves the classical formula for the Gini mean difference, and describes
the reproducibility materials. The main part concludes with open questions in
Section~\ref{sec:conclusion}.

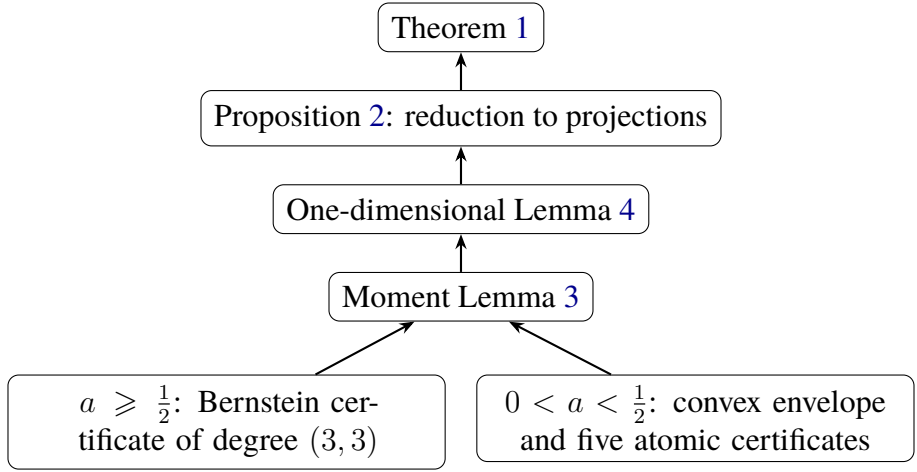
\begin{figure}[ht]
\centering
\begin{tikzpicture}[
 node distance=5mm,
 box/.style={draw,rounded corners,align=center,inner sep=5pt},
 branch/.style={draw,rounded corners,align=center,inner sep=4pt,text width=55mm},
 arr/.style={-{Stealth[length=2mm]},thick}
]
\node[box] (thm) {Theorem~\ref{thm:main}};
\node[box,below=of thm] (red) {Proposition~\ref{prop:projection-reduction}: reduction to projections};
\node[box,below=of red] (one) {One-dimensional Lemma~\ref{lem:one-dimensional}};
\node[box,below=of one] (mom) {Moment Lemma~\ref{lem:moment}};
\node[branch,below=7mm of mom,xshift=-31mm] (large) {$a\geq\tfrac12$: Bernstein certificate of degree $(3,3)$};
\node[branch,below=7mm of mom,xshift=31mm] (small) {$0<a<\tfrac12$: convex envelope and five atomic certificates};
\draw[arr] (red)--(thm);
\draw[arr] (one)--(red);
\draw[arr] (mom)--(one);
\draw[arr] (large)--(mom);
\draw[arr] (small)--(mom);
\end{tikzpicture}
\caption{Logical structure of the proof. Each arrow points from a statement
being used to its consequence.}\label{fig:proof-map}
\end{figure}

\section{Statement of the result}\label{sec:statement}

Throughout, a convex body means a compact convex set $K\subset\R^2$ with
nonempty interior. We write $|K|$ for its area and $P(K)$ for its perimeter.
Let $I$ be uniformly distributed in $K$ with respect to Lebesgue measure, and
let $B$ be uniformly distributed on $\partial K$ with respect to arc length.
Subscripts denote independent copies.

\begin{theorem}[Zaporozhets--Tarasov conjecture]\label{thm:main}
For every planar convex body $K$,
\[
   \E\,|I_1-I_2|<\E\,|B_1-B_2|.
\]
\end{theorem}

\begin{remark}[equality and limiting sharpness]\label{rem:strict-sharpness}
Equality in the theorem is impossible. Nevertheless, the strict inequality
cannot be strengthened to include a universal multiplicative gap. Indeed, let
$K_\delta=[0,1]\times[0,\delta]$, where $\delta\downarrow0$. An interior
point can be written as $I_\delta=(U,\delta V)$, where $U,V\in[0,1]$ are
independent and uniform. By the dominated convergence theorem,
\[
 \E|I_{\delta,1}-I_{\delta,2}|
 \longrightarrow \E|U_1-U_2|=\frac13.
\]
For a boundary point, the probability of lying on one of the two horizontal
sides is $1/(1+\delta)\to1$. Conditional on this pair of sides, its first
coordinate is uniform on $[0,1]$, while the difference of the second coordinates
is at most $\delta$. The probability that at least one of two independent
boundary points lies on a vertical side is at most $2\delta/(1+\delta)$ and
tends to zero. On the complement of this event, the Euclidean distance differs
from the absolute difference of the first coordinates by at most $\delta$, and
all distances are bounded for $0<\delta\leq1$. Hence
\[
 \E|B_{\delta,1}-B_{\delta,2}|
 \longrightarrow \E|U_1-U_2|=\frac13.
\]
Thus, both sides have the common limit $1/3$ as the rectangles collapse to a
line segment, and the ratio of the boundary mean to the interior mean tends to
$1$. A homothety may additionally normalize the diameter of every rectangle to
one. The degenerate limit itself is not a convex body in the sense used here.
\end{remark}

\subsection*{Idea of the proof}

First project the body onto a line in an arbitrary direction for which the
projection has positive length, and then normalize the projection coordinate to
$[0,1]$. It suffices to prove the inequality for the projected area and boundary
length distributions. These distributions turn out to satisfy a nontrivial
relation. More precisely, one can explicitly construct a decreasing probability
density $f$ and an increasing probability density $g$ on $[0,1]$, with
distribution functions $F$ and $G$, respectively, such that the projected
boundary-length density is $(f+g)/2$ and the projected area density is $H/d$,
where
\[
 H=F-G,\qquad d=\int_0^1 H(t)\,\dd t.
\]
Thus, the entire geometric problem reduces to a one-dimensional comparison:
\[
 K\ \longrightarrow\
 \left(\frac{H}{d},\frac{f+g}{2}\right)
 \ \longrightarrow\ \text{comparison of mean distances between projections}.
\]
The inequalities in all directions are then integrated using the identity
\[
 |z|=\frac12\int_0^\pi|\langle z,e_\theta\rangle|\dd\theta.
\]

The key one-dimensional step uses the representation of the Gini mean
difference in terms of the distribution function. It converts the problem into
a comparison of two quadratic \emph{deficits} in the sense of
Definition~\ref{def:gini-deficit}. The deficit of the boundary projection is
bounded from above using Bessel's inequality and the moment lemma, while the
deficit of the area projection is bounded from below by an exact integration by
parts. This is the only point where a nontrivial moment estimate is needed; its
technical proof is deferred to the appendix. The exact formula for the interior
deficit contains a positive quadratic remainder. It cannot vanish because
$H(0)=H(1)=0$ and $d>0$; this yields the strict inequality.

\section{Geometric reduction to a one-dimensional problem}\label{sec:geometry}

We first isolate the geometric part of the proof. It explains why precisely the
densities $H/d$ and $(f+g)/2$ occur in the one-dimensional lemma.

\begin{proposition}[reduction to projections]\label{prop:projection-reduction}
If Lemma~\ref{lem:one-dimensional} holds, then so does
Theorem~\ref{thm:main}.
\end{proposition}

\begin{proof}
Fix a unit vector $e_\theta$ and choose orthonormal coordinates with the
$x$-axis directed along $e_\theta$ and the $y$-axis along $e_\theta^\perp$.
Thus, $x$ is the projection coordinate; below, \enquote{vertical} means
\enquote{parallel to $e_\theta^\perp$}.

Assume that both extreme support faces of $K$, lying on the lines
\[
 x=\min_{z\in K}\langle z,e_\theta\rangle,\qquad
 x=\max_{z\in K}\langle z,e_\theta\rangle,
\]
are singletons. This holds for almost every direction: the only exceptional
directions are those normal to maximal nondegenerate support faces of $K$. The
family of such faces is at most countable, so the corresponding set of
directions has measure zero. This standard fact and its short proof are given in
Appendix~\ref{app:countable-support-segments}. After a translation, let the
projection of $K$ onto the $x$-axis be $[0,\ell]$ and introduce the parameter
$t=x/\ell\in[0,1]$. For each $t$, consider the vertical section
\[
 K_t\eq\{y\in\R:(\ell t,y)\in K\}.
\]
It is nonempty because the projection of $K$ is $[0,\ell]$, compact because
$K$ is compact, and convex because it is the intersection of a convex set with
a line. Hence $K_t$ is a closed interval. Set
\[
 v(t)\eq\min K_t,\qquad u(t)\eq\max K_t.
\]
Then
\[
 K=\{(\ell t,y):0\leq t\leq1,\ v(t)\leq y\leq u(t)\},
\]
and convexity of $K$ implies that $u$ is concave and $v$ is convex. Here
$\ell>0$ is the width of the projection, and
\[
 h(t)\eq u(t)-v(t),\qquad h(0)=h(1)=0.
\]

\begin{figure}[ht]
\centering
\begin{tikzpicture}[x=1.05cm,y=1.05cm,>=Stealth]
\fill[blue!7]
 (0,0) .. controls (1.3,1.65) and (4.7,1.45) .. (6,0)
 .. controls (4.5,-1.25) and (1.4,-1.15) .. (0,0);
\draw[thick,blue!65!black]
 (0,0) .. controls (1.3,1.65) and (4.7,1.45) .. (6,0)
 node[pos=.58,above=3pt] {$u(t)$};
\draw[thick,blue!65!black]
 (0,0) .. controls (1.4,-1.15) and (4.5,-1.25) .. (6,0)
 node[pos=.55,below=3pt] {$v(t)$};
\draw[->] (-.35,-1.55)--(6.6,-1.55) node[right] {$e_\theta$};
\draw[densely dashed] (0,-1.55)--(0,0);
\draw[densely dashed] (6,-1.55)--(6,0);
\node[below] at (0,-1.55) {$0$};
\node[below] at (6,-1.55) {$1$};
\draw[<->,red!70!black,thick] (3.25,-1.02)--(3.25,1.20)
 node[midway,right=3pt] {$h(t)$};
\draw[densely dashed] (3.25,-1.55)--(3.25,-1.02);
\node[below] at (3.25,-1.55) {$t$};
\node at (3.0,.1) {$K$};
\end{tikzpicture}
\caption{Schematic representation of the body between the graphs $v(t)$ and $u(t)$
in a fixed projection direction.}\label{fig:projection-body}
\end{figure}
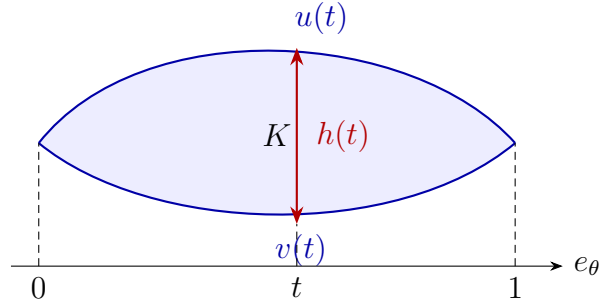

The projection of the uniform area measure has density
\begin{equation}\label{eq:area-projection}
                \rho_{\mathrm{in}}(t)=\frac{\ell h(t)}{|K|}.
\end{equation}
This formula is derived in Appendix~\ref{app:area-projection-density}.

Since $u$ is concave and $v$ is convex, the derivatives $u',v'$ exist almost
everywhere. The projection of normalized boundary length has density
\begin{equation}\label{eq:boundary-projection}
 \rho_{\mathrm{bd}}(t)=\frac{q(t)}{P(K)},\qquad
 q(t)\eq\sqrt{\ell^2+u'(t)^2}+\sqrt{\ell^2+v'(t)^2}.
\end{equation}
This formula is derived in Appendix~\ref{app:boundary-projection-density}.
Define
\begin{equation}\label{eq:rho-pm}
 f(t)\eq\frac{q(t)+h'(t)}{P(K)},\qquad
 g(t)\eq\frac{q(t)-h'(t)}{P(K)}.
\end{equation}
It is readily verified that $s\mapsto\sqrt{\ell^2+s^2}+s$ is strictly
increasing, while $s\mapsto\sqrt{\ell^2+s^2}-s$ is strictly decreasing. Since
$u'$ is decreasing and $v'$ is increasing,
\[
 q+h'=(\sqrt{\ell^2+u'^2}+u')+(\sqrt{\ell^2+v'^2}-v')
\]
is decreasing. Similarly, $q-h'$ is increasing. Moreover,
\[
 \int_0^1q(t)\dd t=P(K),\qquad \int_0^1h'(t)\dd t=0,
\]
so $f$ and $g$ are probability densities, decreasing and increasing,
respectively. It follows from \eqref{eq:boundary-projection}--
\eqref{eq:rho-pm}
\[
             \rho_{\mathrm{bd}}=\frac{f+g}{2}.
\]

If $F$ and $G$ are the distribution functions associated with $f$ and $g$,
then
\begin{equation}\label{eq:H-geometric}
 H(t)=F(t)-G(t)=\frac{2h(t)}{P(K)},\qquad
 d=\int_0^1H(t)\dd t=\frac{2|K|}{P(K)\ell}>0.
\end{equation}
Consequently,
\[
                   \frac{H(t)}d=\frac{\ell h(t)}{|K|}
                   =\rho_{\mathrm{in}}(t).
\]
Lemma~\ref{lem:one-dimensional} applies and, in the present notation, gives
\[
 \iint_{[0,1]^2}|t_1-t_2|
 \frac{H(t_1)}d\frac{H(t_2)}d\dd t_1\dd t_2
 <
 \iint_{[0,1]^2}|t_1-t_2|
 \frac{f(t_1)+g(t_1)}2\frac{f(t_2)+g(t_2)}2\dd t_1\dd t_2.
\]
The left-hand side is the mean distance between the parameters of two
independent points distributed according to the projection of the interior
measure; the right-hand side is the corresponding mean for the projection of
the boundary measure. Proving this one-dimensional inequality contains the
entire technical difficulty. Only a geometric rescaling and averaging over
directions remain.

For two points with parameters $t_1,t_2$, the distance between their projection
coordinates is
\[
 \bigl|\langle z_1-z_2,e_\theta\rangle\bigr|=\ell|t_1-t_2|.
\]
Thus, multiplying the strict one-dimensional inequality by $\ell>0$ yields,
for almost every direction $\theta$,
\begin{equation}\label{eq:directional}
 \E\left|\langle I_1-I_2,e_\theta\rangle\right|
 <
 \E\left|\langle B_1-B_2,e_\theta\rangle\right|.
\end{equation}

All statements about derivatives above are understood almost everywhere; the
arc-length formulas and the monotonicity of the derivatives of a convex function
are standard in this generality. In an exceptional direction, one of the
extreme support faces at $x=0$ or $x=\ell$ may be a nontrivial vertical segment
of the boundary. The corresponding set of values of $\theta$ is at most
countable and therefore has measure zero in the integration over $\theta$.

Finally, for each $z\in\R^2$
\begin{equation}\label{eq:cosine-representation}
       |z|=\frac12\int_0^\pi|\langle z,e_\theta\rangle|\dd\theta.
\end{equation}
This identity follows by rotating coordinates: for
$z=(|z|,0)$ the integral is equal to $|z|\int_0^\pi|\cos\theta|\dd\theta=2|z|$.
The difference between the right- and left-hand sides of
\eqref{eq:directional} is a measurable function of $\theta$ and is strictly
positive almost everywhere. Its integral is therefore strictly positive.
Integrating over $\theta$ and using Fubini's theorem and
\eqref{eq:cosine-representation}, we get
\[
 \E|I_1-I_2|
 =\frac12\int_0^\pi\E|\langle I_1-I_2,e_\theta\rangle|\dd\theta
 <
 \frac12\int_0^\pi\E|\langle B_1-B_2,e_\theta\rangle|\dd\theta
 =\E|B_1-B_2|,
\]
as required.
\end{proof}

\section{One-dimensional comparison}\label{sec:one-dimensional}

We now turn to the analytic core of the paper. The only technical input to the
one-dimensional proof is the following moment estimate. Its proof is given in
Appendix~\ref{app:moment-proof}.

Set
\[
       \gamma(r)\eq1+r-2r^2=(1-r)(1+2r),\qquad 0\leq r\leq1.
\]

\begin{lemma}[moment lemma]\label{lem:moment}
Let $R\in[0,1]$, $\varepsilon\in\{-1,1\}$ and
\[
  \mathbb P\{\varepsilon=1\}=\mathbb P\{\varepsilon=-1\}=\frac12;
\]
the dependence between $R$ and $\varepsilon$ is arbitrary. If
\[
 a=\E R>0,\qquad b=\E\bigl[R^2(1-R)\bigr],
\]
then
\begin{equation}\label{eq:moment-main}
 \frac14\Var(\varepsilon R)+\frac1{12}\Var(\gamma(R))
 \geq\frac{a^2-b}{6a}.
\end{equation}
\end{lemma}

\begin{samepage}
The one-dimensional lemma compares $\E|X_1-X_2|$ for two different
distributions. To write both sides uniformly as functionals of the corresponding
measures, we introduce the following notation.

\begin{definition}\label{def:gini}
For a probability measure $\sigma$ on $[0,1]$, set
\[
 D(\sigma)\eq\iint|x-y|\,\sigma(\dd x)\sigma(\dd y).
\]
\end{definition}
\end{samepage}

If $X,Y$ are independent and both have distribution $\sigma$, then
$D(\sigma)=\E|X-Y|$, the mean absolute difference of two independent
observations. This quantity is therefore called the \emph{Gini mean
difference}. Gini introduced mean pairwise differences as a measure of
variability in~\cite[p.~22]{Gini1912}.

The word \enquote{unstandardized} refers not to the measure $\sigma$, which
still has total mass $1$, but to the quantity $D(\sigma)$ itself: it is not
divided by any scale characteristic of the distribution. In particular, replacing
$X$ by $cX$, where $c>0$, multiplies the mean difference by $c$, so it has the
same physical dimension as $X$.

\begin{definition}[Gini coefficient]\label{def:gini-coefficient}
For the probability measure $\sigma$ on $[0,1]$ with positive mean
\[
 m_\sigma\eq\int_0^1x\,\sigma(\dd x)>0
\]
the quantity
\[
 \operatorname{Gin}(\sigma)\eq\frac{D(\sigma)}{2m_\sigma}
\]
is called the \emph{Gini coefficient}.
\end{definition}

The Gini coefficient is dimensionless and is unchanged when $X$ is replaced by
$cX$. This relative normalization was introduced by Gini as the concentration
ratio in~\cite{Gini1914}. The present proof requires $D(\sigma)$ rather than
$\operatorname{Gin}(\sigma)$, because we compare the mean distances themselves,
not their ratios to the means of the distributions.

If $S$ is the distribution function of $\sigma$, then the classical
representation
\begin{equation}\label{eq:gini-cdf}
 D(\sigma)=2\int_0^1S(x)(1-S(x))\dd x
 =\frac12-2\int_0^1(S(x)-\tfrac12)^2\dd x.
\end{equation}
This is not a new result of the present paper: a standard layer-cake derivation
is given, for example, in
\cite[Section~2.1.2, formulas~(2.4)--(2.9)]{YitzhakiSchechtman2013}.
A short equivalent proof using indicator functions and Tonelli's theorem is
given in Appendix~\ref{app:gini-cdf-proof}.

For the uniform probability measure $\lambda$ on $[0,1]$, the distribution
function is $S_\lambda(x)=x$, hence
\[
 \int_0^1(S_\lambda(x)-\tfrac12)^2\dd x
 =\int_0^1(x-\tfrac12)^2\dd x=\frac1{12},
 \qquad D(\lambda)=\frac13.
\]

To prove the one-dimensional lemma, it is convenient to replace the comparison
of Gini mean differences by the reverse comparison of the corresponding
quadratic deviations, which can be estimated using Bessel's inequality and
integration by parts. In the theory of sharp inequalities, the difference
between a reference value and the actual value is commonly called a
\emph{deficit}; its precise normalization depends on the problem.

\begin{definition}[deficit relative to uniform distribution]
\label{def:gini-deficit}
For a probability measure $\sigma$ on $[0,1]$ with distribution function $S$,
set
\begin{equation}\label{eq:gini-deficit}
 E(\sigma)\eq
 \int_0^1(S(x)-\tfrac12)^2\dd x-\frac1{12}
 =\frac{1/3-D(\sigma)}2.
\end{equation}
We call $E(\sigma)$ the \emph{deficit of the Gini mean difference relative to
the uniform distribution}. For an arbitrary $\sigma$ it may have either sign;
the word \enquote{deficit} indicates the chosen reference distribution and does
not assert nonnegativity in advance.
\end{definition}

\begin{lemma}[one-dimensional lemma]\label{lem:one-dimensional}
Let $f$ be a decreasing probability density and $g$ an increasing probability
density on $[0,1]$. Let $F,G$ be their distribution functions,
\[
 H=F-G,\qquad d=\int_0^1H(x)\dd x>0.
\]
Define
\[
 \sigma_{\mathrm{in}}(\dd x)\eq\frac{H(x)}d\dd x,\qquad
 \sigma_{\mathrm{bd}}(\dd x)\eq\frac{f(x)+g(x)}2\dd x
\]
and set
\[
 D_{\mathrm{in}}\eq D(\sigma_{\mathrm{in}}),\qquad
 D_{\mathrm{bd}}\eq D(\sigma_{\mathrm{bd}}),\qquad
 E_{\mathrm{in}}\eq E(\sigma_{\mathrm{in}}),\qquad
 E_{\mathrm{bd}}\eq E(\sigma_{\mathrm{bd}}).
\]
Then the following strict inequalities are equivalent:
\begin{equation}\label{eq:one-dimensional-result}
 D_{\mathrm{in}}<D_{\mathrm{bd}}
 \qquad\Longleftrightarrow\qquad
 E_{\mathrm{in}}>E_{\mathrm{bd}}.
\end{equation}
\end{lemma}

\begin{proof}
\step{Step 1: validity of the density $H/d$.}
Temporarily assume that $f$ and $g$ are bounded; this assumption will be
removed in the final step.

Fix an interior point $0<x<1$. Since $f$ is a probability density, its mean
value over the whole interval $[0,1]$ is $\int_0^1f(t)\dd t=1$. Because $f$
is decreasing, its mean over the left interval $[0,x]$ is at least its mean
over the right interval $[x,1]$:
\[
 \frac{F(x)}x=\frac1x\int_0^x f(t)\dd t
 \geq \frac1{1-x}\int_x^1f(t)\dd t
 =\frac{1-F(x)}{1-x}.
\]
Multiplying by $x(1-x)>0$ shows that this inequality is equivalent to
$F(x)\geq x$. For the increasing density $g$, the inequality is reversed:
\[
 \frac{G(x)}x\leq\frac{1-G(x)}{1-x},
 \qquad\text{hence}\qquad G(x)\leq x.
\]
At $x=0$ and $x=1$, the corresponding equalities follow from
$F(0)=G(0)=0$ and $F(1)=G(1)=1$. Thus
$F(x)\geq x\geq G(x)$ on $[0,1]$, and hence $H=F-G\geq0$.
If $d=0$, continuity and nonnegativity of $H$ would imply $H\equiv0$, and
$H/d$ would be undefined; this case is excluded by the assumption $d>0$ in
the lemma. Thus $H/d$ is well defined, and the definition of $d$ gives
$\int_0^1H(x)/d\,\dd x=1$.

\step{Step 2: representation of monotone densities as mixtures.}
\phantomsection\label{step:mixture-representation}
The next idea is to represent $f$ and $g$ as \emph{mixtures of uniform
distributions} on one-sided intervals. We denote the variability of the
component distribution functions by $\mathcal V$. Bessel's inequality and the
moment lemma will give an upper bound for the boundary deficit, while the
curvature $-H''$ will give an exact lower bound for the interior deficit.

We use Khinchin's classical representation of unimodal distributions as
mixtures of uniform distributions~\cite{Khinchin1938}. In the present one-sided
case it follows directly. For clarity, first assume that $f$ is differentiable,
and denote by
\[
 u_s(x)\eq\frac{\mathds{1}_{[0,s]}(x)}s,\qquad 0<s\leq1,
\]
the density of the uniform distribution on $[0,s]$. Since $f$ is decreasing,
$-f'(s)\geq0$, and
\begin{align*}
 f(x)
 &=f(1)+\int_x^1\bigl(-f'(s)\bigr)\dd s\\
 &=f(1)u_1(x)
   +\int_0^1u_s(x)\,s\bigl(-f'(s)\bigr)\dd s.
\end{align*}
The last line has an exact probabilistic interpretation: $f$ is a mixture of
the densities $u_s$, with weight $s(-f'(s))\dd s$ assigned to the component
$U[0,s]$, while the component $U[0,1]$ also receives the weight $f(1)$.
Integration by parts shows that the total mass of these weights is
\[
 f(1)+\int\limits_0^1s\bigl(-f'(s)\bigr)\dd s
 =\int\limits_0^1f(s)\dd s=1.
\]
Thus this is a genuine probability mixture, not merely a formal linear
combination.

For an arbitrary bounded decreasing $f$, the derivative is replaced by the
nonnegative Stieltjes measure $\alpha=-\dd f$. The identities are understood at
points of continuity of $f$, and hence almost everywhere. Then
\[
 f(x)=f(1)+\alpha([x,1])
     =f(1)u_1(x)+\int_{(0,1)}u_s(x)\,s\alpha(\dd s),
\]
and the mixing probability measure is
\[
 \pi_f(\dd s)\eq s\alpha(\dd s)+f(1)\delta_1(\dd s),
 \qquad
 \pi_f((0,1])=f(1)+\int_{(0,1)}s\alpha(\dd s)=1.
\]

The argument for increasing $g$ is symmetric. If
\[
 v_s(x)\eq\frac{\mathds{1}_{[s,1]}(x)}{1-s},\qquad 0\leq s<1,
\]
then in the differentiable case
\[
 g(x)=g(0)v_0(x)
      +\int_0^1v_s(x)(1-s)g'(s)\dd s.
\]
In the general case, setting $\beta=\dd g$ gives a mixture of uniform
distributions on $[s,1]$ with mixing probability measure
\[
 \pi_g(\dd s)\eq(1-s)\beta(\dd s)+g(0)\delta_0(\dd s),
 \qquad \pi_g([0,1))=1.
\]
In other words, a random variable with density $f$ can be generated by first
choosing $S$ with law $\pi_f$ and then choosing a point uniformly on $[0,S]$.
For $g$, after choosing $S$ with law $\pi_g$, the point is chosen uniformly on
$[S,1]$.

Mix the two distributions with equal weights; the mixture has density
$(f+g)/2$. For uniform notation, introduce a random orientation
$\varepsilon\in\{-1,1\}$ taking both values with equal probabilities.
Conditionally on $\varepsilon$, choose $R\in[0,1]$ so that the law of $1-R$ is
$\pi_f$ when $\varepsilon=1$, and the law of $R$ is $\pi_g$ when
$\varepsilon=-1$. Conditionally on $(\varepsilon,R)$, choose $X$ according to
the probability measure $\mu_{\varepsilon,R}$, where, for $0\leq r<1$,
\[
 \mu_{\varepsilon,r}\eq
 \begin{cases}
 U[0,1-r],&\varepsilon=1,\\
 U[r,1],&\varepsilon=-1.
 \end{cases}
\]
Here $U[I]$ denotes the uniform probability measure on the interval $I$. For
$r=1$, set
$\mu_{1,1}=\delta_0$ and $\mu_{-1,1}=\delta_1$. Measures
$\mu_{\varepsilon,r}$ will be called \emph{mixture components}: they are the
conditional laws of $X$ for a fixed selecting pair $(\varepsilon,r)$. Let
$C_{\varepsilon,r}$ denote the distribution function of
$\mu_{\varepsilon,r}$. For $r<1$, its restriction to $[0,1]$ is
\[
 C_{1,r}(x)=
 \begin{cases}
 x/(1-r),&0\leq x\leq1-r,\\
 1,&1-r<x\leq1,
 \end{cases}
 \qquad
 C_{-1,r}(x)=
 \begin{cases}
 0,&0\leq x<r,\\
 (x-r)/(1-r),&r\leq x\leq1.
 \end{cases}
\]

\begin{remark}
For fixed $(\varepsilon,r)$, the function $C_{\varepsilon,r}$ is
deterministic. The function $C_{\varepsilon,R}$ is random only because the pair
$(\varepsilon,R)$ selecting the component is random. In particular, for fixed
$x$, the value $C_{\varepsilon,R}(x)$ is an ordinary random variable, and all
expectations and variances below are taken with respect to $(\varepsilon,R)$.
\end{remark}

By the law of total probability, the mixture of these components has density
$(f+g)/2$, and its distribution function is
\[
\begin{aligned}
 M(x)
 &\eq \mathbb P(X\leq x)\\
 &=\E\!\left[\mathbb P(X\leq x\mid\varepsilon,R)\right]\\
 &=\E C_{\varepsilon,R}(x).
\end{aligned}
\]
Set
\[
 a\eq\E R,\qquad b\eq\E[R^2(1-R)].
\]
Denote the variance of the component distribution functions, averaged over
$x$, by
\[
 \mathcal V\eq\int_0^1\Var(C_{\varepsilon,R}(x))\dd x.
\]
The quantity $\mathcal V$ measures the average variability of the component
distribution functions around their mean function $M$. Equivalently, by the
definition of variance and Tonelli's theorem,
\[
 \mathcal V
 =\E\int_0^1\bigl(C_{\varepsilon,R}(x)-M(x)\bigr)^2\dd x.
\]

\begin{remark}[role of the mixture representation]
The representation serves three purposes. First, it replaces arbitrary
monotone densities by one-parameter uniform components with explicitly
computable piecewise-linear distribution functions. Second, the
infinite-dimensional data in $f$ and $g$ are encoded by the pair
$(\varepsilon,R)$ and the quantities $a,b,\mathcal V$, to which Bessel's
inequality and the moment lemma apply. Third, the same mixing measures describe
the curvature $-H''=\alpha+\beta$, linking the boundary density $(f+g)/2$ to
the interior density $H/d$. Schematically, the remainder of the proof is
\[
 (f,g)\ \longrightarrow\ (\varepsilon,R)
 \ \longrightarrow\ (a,b,\mathcal V)
 \ \longrightarrow\
 D_{\mathrm{in}}<D_{\mathrm{bd}}.
\]
The representation is chosen precisely because it connects both measures being
compared, not merely because its components are simple.
\end{remark}

\step{Step 3: formulas for the components.}
For each component, direct integration gives
\begin{align}
 \int_0^1C_{\varepsilon,r}(x)\dd x
   &=\frac{1+\varepsilon r}{2},\label{eq:component-mode0}\\
 \int_0^1C_{\varepsilon,r}(x)(x-\tfrac12)\dd x
   &=\frac{1+r-2r^2}{12}
     =\frac{\gamma(r)}{12},
     \qquad \gamma(r)\eq1+r-2r^2,
     \label{eq:component-mode1}\\
 D(\mu_{\varepsilon,r})&=\frac{1-r}{3}.
 \label{eq:component-gini}
\end{align}
Both component measures are uniform on intervals of the same length $1-r$.
For the uniform distribution on an arbitrary interval of length $L$, the mean
absolute difference of two independent random points is $L/3$. Here
$L=1-r$, which gives
\eqref{eq:component-gini}.

Averaging \eqref{eq:component-mode0} separately conditional on
$\varepsilon=1$ and $\varepsilon=-1$, and using that these events have equal
probabilities, gives
\begin{equation}\label{eq:d-equals-a}
 d=\int_0^1(F-G)\dd x
   =\frac12\bigl(\E[R\mid\varepsilon=1]
                 +\E[R\mid\varepsilon=-1]\bigr)
   =\E R=a>0.
\end{equation}

We now apply the law of total variance in detail. For fixed $x\in[0,1]$, set
$I_x\eq\mathds{1}_{\{X\leq x\}}$. Then
\[
 \E[I_x\mid\varepsilon,R]=C_{\varepsilon,R}(x),
 \qquad
 \E I_x=M(x).
\]
Since $I_x$ is an indicator, the law of total variance gives
\[
\begin{aligned}
 M(x)(1-M(x))
 &=\Var(I_x)\\
 &=\E\!\left[\Var(I_x\mid\varepsilon,R)\right]
   +\Var\!\left(\E[I_x\mid\varepsilon,R]\right)\\
 &=\E\!\left[
    C_{\varepsilon,R}(x)(1-C_{\varepsilon,R}(x))\right]
   +\Var(C_{\varepsilon,R}(x)).
\end{aligned}
\]
Integrate this identity over $x$ and use \eqref{eq:gini-cdf}. After
multiplication by $2$, the left-hand side is the Gini mean difference of the
entire mixture, namely $D_{\mathrm{bd}}$. The first term on the right is the
average Gini mean difference within a randomly selected component. Therefore
\begin{align}
 D_{\mathrm{bd}}
 &=2\int_0^1M(x)(1-M(x))\dd x \notag\\
 &=\E\!\left[
     2\int_0^1C_{\varepsilon,R}(x)
       (1-C_{\varepsilon,R}(x))\dd x\right]
   +2\int_0^1\Var(C_{\varepsilon,R}(x))\dd x \notag\\
 &=\E D(\mu_{\varepsilon,R})+2\mathcal V
  =\frac{1-a}{3}+2\mathcal V.
 \label{eq:mixture-gini}
\end{align}
Here $\E D(\mu_{\varepsilon,R})=(1-a)/3$ is the mean within-component
Gini difference: first choose one random component, and then independently
sample two points from it. The additional term $2\mathcal V$ accounts exactly
for the \emph{between-component variability} that appears in the Gini mean
difference of the full mixture. Thus $\mathcal V$ is not an artificial
auxiliary quantity, but the exact correction that converts the mean
within-component Gini difference into the Gini mean difference of the mixture.
It is this correction that must next be bounded from below.

\step{Step 4: Bessel's inequality and an estimate for $\mathcal V$.}
Recall the following form of Bessel's inequality
\cite{KolmogorovFomin1972}: if $\varphi_0,\varphi_1$ are orthonormal in a
Hilbert space, then every element $h$ satisfies
\[
 \|h\|^2\geq |\langle h,\varphi_0\rangle|^2
                  +|\langle h,\varphi_1\rangle|^2.
\]
This follows directly by decomposing $h$ into its projections onto
$\varphi_0,\varphi_1$ and a remainder orthogonal to both.

In $L^2[0,1]$, the functions
\[
 \varphi_0(x)=1,\qquad \varphi_1(x)=\sqrt{12}(x-\tfrac12)
\]
are orthonormal. Set
$Z_{\varepsilon,R}(x)\eq C_{\varepsilon,R}(x)-M(x)$. For every fixed value of
$(\varepsilon,R)$, apply Bessel's inequality to $Z_{\varepsilon,R}$.

We compute the two projections explicitly. Since
$M(x)=\E C_{\varepsilon,R}(x)$, Fubini's theorem gives, for every bounded
function $\psi$,
\[
 \int_0^1M(x)\psi(x)\dd x
 =\E\int_0^1C_{\varepsilon,R}(x)\psi(x)\dd x.
\]
Taking $\varphi_0(x)=1$ and using \eqref{eq:component-mode0}, we obtain
\[
\begin{aligned}
 \langle Z_{\varepsilon,R},\varphi_0\rangle
 &=\int_0^1\bigl(C_{\varepsilon,R}(x)-M(x)\bigr)\dd x\\
 &=\frac{1+\varepsilon R}{2}
   -\E\!\left[\frac{1+\varepsilon R}{2}\right]\\
 &=\frac{\varepsilon R-\E(\varepsilon R)}2.
\end{aligned}
\]
Similarly, for $\varphi_1(x)=\sqrt{12}(x-\tfrac12)$,
\eqref{eq:component-mode1} gives
\[
\begin{aligned}
 \langle Z_{\varepsilon,R},\varphi_1\rangle
 &=\sqrt{12}\int_0^1
   \bigl(C_{\varepsilon,R}(x)-M(x)\bigr)(x-\tfrac12)\dd x\\
 &=\sqrt{12}\left(
     \frac{\gamma(R)}{12}
     -\E\!\left[\frac{\gamma(R)}{12}\right]\right)\\
 &=\frac{\gamma(R)-\E\gamma(R)}{\sqrt{12}}.
\end{aligned}
\]
The second projection does not involve $\varepsilon$ because
\eqref{eq:component-mode1} has the same value for both orientations of the
component.

We also compute the mean of the left-hand side of Bessel's inequality. By the
definition of $M$ and Fubini's theorem,
\[
 \E\|Z_{\varepsilon,R}\|_2^2
 =\int\limits_0^1\E\bigl(C_{\varepsilon,R}(x)-M(x)\bigr)^2\dd x
 =\int\limits_0^1\Var(C_{\varepsilon,R}(x))\dd x
 =\mathcal V.
\]
Now average Bessel's inequality over $(\varepsilon,R)$. The means of the
squares of the two projections are, respectively,
$\frac14\Var(\varepsilon R)$ and $\frac1{12}\Var(\gamma(R))$. Hence
\begin{equation}\label{eq:Bessel-V}
 \mathcal V\geq\frac14\Var(\varepsilon R)
                   +\frac1{12}\Var(\gamma(R)).
\end{equation}
The right-hand side of \eqref{eq:Bessel-V} is precisely the left-hand side of
the inequality in the moment Lemma~\ref{lem:moment}. Thus that lemma gives
\[
 \frac14\Var(\varepsilon R)+\frac1{12}\Var(\gamma(R))
 \geq\frac{a^2-b}{6a}.
\]
Substituting this lower bound into \eqref{eq:Bessel-V}, we obtain
\begin{equation}\label{eq:V-lower}
              \mathcal V\geq\frac{a^2-b}{6a}.
\end{equation}

\step{Step 5: deficit of the boundary measure.}
The function $M$ is the distribution function of $\sigma_{\mathrm{bd}}$.
Therefore, Definition~\ref{def:gini-deficit} and
\eqref{eq:mixture-gini} give
\begin{equation}\label{eq:boundary-deficit}
 E_{\mathrm{bd}}=\int_0^1(M(x)-\tfrac12)^2\dd x-\frac1{12}
 =\frac{1/3-D_{\mathrm{bd}}}2
 =\frac a6-\mathcal V
 \leq\frac{b}{6a}.
\end{equation}
Thus, the deficit of the boundary measure is bounded from above by two explicit
moment characteristics, $a$ and $b$. This is exactly the form needed below:
for the interior measure we will obtain the reverse bound with the same
intermediate quantity $b/(6a)$.

\step{Step 6: internal measure deficit.}
It remains to compute the deficit of the density $H/d$. Recall from
\hyperref[step:mixture-representation]{Step~2} that
\[
                 \alpha\eq-\dd f,\qquad \beta\eq\dd g.
\]
These are finite nonnegative Stieltjes measures on $(0,1)$ because $f$ is
decreasing and $g$ is increasing. Every locally integrable function has
derivatives of all orders in the sense of distributions. Here $F$ and $G$ are
absolutely continuous, with $F'=f$ and $G'=g$ almost everywhere. Therefore,
for $H=F-G$, in the sense of distributions on $(0,1)$,
\[
 H'=f-g,\qquad
 H''=\dd f-\dd g=-\alpha-\beta.
\]
Moreover, $H(0)=H(1)=0$, since both $F$ and $G$ take the values $0$ and $1$
at the endpoints of the interval. Thus
\[
 \kappa\eq\alpha+\beta=-H''\geq0,\qquad H(0)=H(1)=0.
\]
The second derivative $H''$ need not exist pointwise: it is a distributional
derivative represented by a finite measure.

Set $A(s)\eq s(1-s)$. Recall also the mixing measures:
\[
 \pi_f(\dd s)=s\alpha(\dd s)+f(1)\delta_1(\dd s),\qquad
 \pi_g(\dd s)=(1-s)\beta(\dd s)+g(0)\delta_0(\dd s).
\]
Conditionally on $\varepsilon=1$, the variable $S=1-R$ has law $\pi_f$;
conditionally on $\varepsilon=-1$, the variable $R$ itself has law $\pi_g$.
Therefore
\begin{align*}
 \E[R\mid\varepsilon=1]
   &=\int_0^1(1-s)\,\pi_f(\dd s)
     =\int_0^1A(s)\,\alpha(\dd s),\\
 \E[R\mid\varepsilon=-1]
   &=\int_0^1s\,\pi_g(\dd s)
     =\int_0^1A(s)\,\beta(\dd s),
\end{align*}
Similarly, for the second moment under consideration,
\begin{align*}
 \E[R^2(1-R)\mid\varepsilon=1]
   &=\int_0^1(1-s)^2s\,\pi_f(\dd s)
     =\int_0^1A(s)^2\,\alpha(\dd s),\\
 \E[R^2(1-R)\mid\varepsilon=-1]
   &=\int_0^1s^2(1-s)\,\pi_g(\dd s)
     =\int_0^1A(s)^2\,\beta(\dd s).
\end{align*}
The atoms $f(1)\delta_1$ and $g(0)\delta_0$ do not contribute because the
corresponding factors vanish at $1$ and $0$. Averaging the conditional
expectations with equal weights and using the definitions $a=\E R$ and
$b=\E[R^2(1-R)]$, together with $d=a$ from
\eqref{eq:d-equals-a}, gives
\begin{equation}\label{eq:a-b-curvature}
 d=a=\frac12\int_0^1A(s)\,\kappa(\dd s),
 \qquad
 b=\frac12\int_0^1A(s)^2\,\kappa(\dd s).
\end{equation}

Let
\[
       N(x)\eq\frac1a\int_0^xH(u)\dd u
\]
be the distribution function of $\sigma_{\mathrm{in}}$, since $d=a$. Set
$Z(x)\eq N(x)-\tfrac12$ and
\[
       I\eq\int_0^1A(s)^2\,\kappa(\dd s)=2b.
\]
Since $\kappa=-H''$, and both the polynomial $A(x)^2$ and its first derivative
vanish at the endpoints,
\begin{align*}
 I&=-\int_0^1\bigl(A(x)^2\bigr)''H(x)\dd x\\
  &=\int_0^1\bigl(1-12(x-\tfrac12)^2\bigr)H(x)\dd x.
\end{align*}
Since $Z'(x)=H(x)/a$, $Z(0)=-1/2$, and $Z(1)=1/2$, integration by parts gives
\[
 \frac{I}{a}=24\int_0^1(x-\tfrac12)Z(x)\dd x-2.
\]
For brevity, set $t(x)=x-\tfrac12$. The last identity implies
\[
 2\int_0^1t(x)Z(x)\dd x=\frac{I}{12a}+\frac16,
 \qquad
 \int_0^1t(x)^2\dd x=\frac1{12}.
\]
Now expand the definition of $E_{\mathrm{in}}$. Since
$Z(x)=(N(x)-x)+t(x)$, we get
\begin{equation}\label{eq:interior-deficit-identity}
 \begin{aligned}
 E_{\mathrm{in}}
 &=\int_0^1Z(x)^2\dd x-\frac1{12}\\
 &=\int_0^1\bigl((N(x)-x)+t(x)\bigr)^2\dd x
   -\int_0^1t(x)^2\dd x\\
 &=\int_0^1\bigl(N(x)-x\bigr)^2\dd x
   +2\int_0^1t(x)Z(x)\dd x-2\int_0^1t(x)^2\dd x\\
 &=\frac{I}{12a}+\int_0^1\bigl(N(x)-x\bigr)^2\dd x\\
 &\geq\frac{I}{12a}=\frac{b}{6a}.
 \end{aligned}
\end{equation}
The last inequality follows from nonnegativity of the integral of a square, and
the last equality from $I=2b$; moreover, $a>0$ by
\eqref{eq:d-equals-a}.

\step{Step 7: comparison and passage to the limit.}
Combining \eqref{eq:boundary-deficit} and
\eqref{eq:interior-deficit-identity}, and writing
\[
 Q\eq\int_0^1\bigl(N(x)-x\bigr)^2\dd x,
\]
we obtain the sharper estimate
\begin{equation}\label{eq:strict-deficit-gap}
 E_{\mathrm{in}}-E_{\mathrm{bd}}\geq Q.
\end{equation}
In fact $Q>0$. Indeed, if $Q=0$, continuity of $N$ would imply $N(x)=x$
throughout $[0,1]$. Differentiating almost everywhere would give
$H(x)/a=1$ almost everywhere. Since $H=F-G$ is continuous, it would follow
that $H(x)=a$ for every $x$. But $H(0)=H(1)=0$, so $a=0$, contradicting
\eqref{eq:d-equals-a}. Therefore
\[
 E_{\mathrm{in}}>E_{\mathrm{bd}}.
\]
By the equivalence in the statement of the lemma, this means
$D_{\mathrm{in}}<D_{\mathrm{bd}}$, namely
\eqref{eq:one-dimensional-result}.

We now remove the boundedness assumption. For arbitrary integrable $f,g$, set
\[
 c_n\eq\int_0^1(f\wedge n)\dd x,\qquad
 \widetilde c_n\eq\int_0^1(g\wedge n)\dd x,
 \qquad
 f_n\eq\frac{f\wedge n}{c_n},\qquad
 g_n\eq\frac{g\wedge n}{\widetilde c_n}.
\]
Here $f\wedge n=\min\{f,n\}$. By the monotone convergence theorem,
$c_n,\widetilde c_n\to1$. Hence $f_n$ and $g_n$ are bounded probability
densities with the same monotonicity as $f$ and $g$, and $f_n\to f$ and
$g_n\to g$ in $L^1[0,1]$, because
\[
 \|f_n-f\|_1\leq2(1-c_n)\longrightarrow0,\qquad
 \|g_n-g\|_1\leq2(1-\widetilde c_n)\longrightarrow0.
\]

If $F_n,G_n$ are their distribution functions,
$H_n=F_n-G_n$, and $d_n=\int_0^1H_n$, then
\[
 \|F_n-F\|_\infty\leq\|f_n-f\|_1,\qquad
 \|G_n-G\|_\infty\leq\|g_n-g\|_1.
\]
Therefore $H_n\to H$ uniformly, $d_n\to d>0$, and
\[
 \frac{H_n}{d_n}\longrightarrow\frac Hd
 \quad\text{uniformly},\qquad
 \frac{f_n+g_n}{2}\longrightarrow\frac{f+g}{2}
  \quad\text{in }L^1[0,1].
\]
For all sufficiently large $n$ we set
\[
 N_n(x)\eq\frac1{d_n}\int_0^xH_n(u)\dd u,\qquad
 Q_n\eq\int_0^1\bigl(N_n(x)-x\bigr)^2\dd x.
\]
For the original densities, define
\[
 N(x)\eq\frac1d\int_0^xH(u)\dd u,\qquad
 Q\eq\int_0^1\bigl(N(x)-x\bigr)^2\dd x.
\]
The argument above proving $Q>0$ does not use boundedness of $f,g$, so it
applies here as well. Moreover, $N_n\to N$ uniformly, and hence $Q_n\to Q$.
For probability densities $p_n,p$ on $[0,1]$ we have
\[
\bigl|D(p_n(x)\dd x)-D(p(x)\dd x)\bigr|
 \leq2\|p_n-p\|_1,
\]
since $|x-y|\leq1$. Consequently, the corresponding quantities
$E_{\mathrm{in},n},E_{\mathrm{bd},n}$ also converge to
$E_{\mathrm{in}},E_{\mathrm{bd}}$. For the bounded densities $f_n,g_n$,
\eqref{eq:strict-deficit-gap} becomes
\[
 E_{\mathrm{in},n}-E_{\mathrm{bd},n}\geq Q_n.
\]
Passing to the limit gives
\[
 E_{\mathrm{in}}-E_{\mathrm{bd}}\geq Q>0.
\]
Thus \eqref{eq:one-dimensional-result} is proved for arbitrary
$f,g$.
\end{proof}

\begin{remark}[role of the boundedness assumption]
The boundedness of $f$ and $g$ is needed only in Steps~2 and~6. Since bounded
monotone functions have finite total variation, the Stieltjes measures
\[
 \alpha=-\dd f,\qquad \beta=\dd g
\]
are finite. This allows us to treat
\[
 \pi_f(\dd s)=s\alpha(\dd s)+f(1)\delta_1(\dd s),\qquad
 \pi_g(\dd s)=(1-s)\beta(\dd s)+g(0)\delta_0(\dd s)
\]
as finite mixing measures without further qualification and to apply Tonelli's
theorem. In Step~6 it also follows that
$\kappa=\alpha+\beta=-H''$ is a finite measure, so the integrals
\[
 \int_0^1A(s)\,\kappa(\dd s),\qquad
 \int_0^1A(s)^2\,\kappa(\dd s)
\]
are immediately well defined, as is the integration by parts used there.

If $f$ is unbounded, it may tend to infinity near $0$; an unbounded $g$ may
behave similarly near $1$. Then $\alpha$ or $\beta$ may have infinite total
mass, even though the weighted measures $s\alpha(\dd s)$ and
$(1-s)\beta(\dd s)$ remain finite. A direct proof would require separate
control near the endpoints. Step~7 avoids this technical issue using the
normalized truncations $f_n,g_n$: they are bounded, and their $L^1$ convergence
permits passage to the original $f,g$. Thus boundedness is only an auxiliary
technical assumption.
\end{remark}

\section{Completion of the proof}\label{sec:completion}

\begin{proof}[Proof of Theorem~\ref{thm:main}]
The moment Lemma~\ref{lem:moment}, proved in
Appendix~\ref{app:moment-proof}, gives estimate~\eqref{eq:V-lower}, from which
the one-dimensional Lemma~\ref{lem:one-dimensional} follows.
Proposition~\ref{prop:projection-reduction} applies that lemma to almost every
projection, and identity~\eqref{eq:cosine-representation} integrates the
resulting strict inequalities over all directions. Thus
$\E|I_1-I_2|<\E|B_1-B_2|$.
\end{proof}

\section{Conclusion}\label{sec:conclusion}

The result settles the original conjecture in its strict form. Equality is
impossible for a convex body with nonempty interior: in the one-dimensional
reduction it would force the quadratic remainder $Q$ to vanish, contradicting
$H(0)=H(1)=0$ and $d>0$. At the same time,
Remark~\ref{rem:strict-sharpness} shows that for thin rectangles the difference
between the mean distances tends to zero. Thus the strict inequality is sharp
in the limiting sense: even after normalizing the diameter, there is no
universal positive additive or multiplicative gap.

Several natural generalizations remain open. First, for centrally symmetric
bodies A.~S.~Lotnikov~\cite{Lotnikov2025} proved the comparison simultaneously
for all moments of order $p\geq1$. The principal open question is whether the
same remains true without central symmetry. Our one-dimensional argument relies
essentially on representation~\eqref{eq:gini-cdf}, which is specific to the
first moment, and it is not yet clear how to extend it to higher moments.

Second, the problem itself makes sense in $\R^d$ for $d\geq3$: one may compare
mean distances for the interior and boundary uniform measures of a convex
body. Our method exploits a specifically two-dimensional feature---the
monotonicity of a pair of boundary-projection densities---and a
higher-dimensional generalization would require a new geometric input.

Finally, under the assumption that the centroids coincide,
A.~S.~Tokmachev~\cite{Tokmachev2026} established a stronger comparison for
convex functions. It is natural to ask whether his approach can be combined
with the present method to remove the centroid condition. In any event, these
questions leave room for further research.

\appendix

\section{Proof of the moment lemma}\label{app:moment-proof}

Throughout this appendix we retain the notation of Lemma~\ref{lem:moment}.
The plan is as follows. We first study separately the convex envelope of an
auxiliary quartic, which controls the range $0<a<\tfrac12$. The proof then
proceeds in increasing order of complexity: the range $a\geq\tfrac12$ with a
single Bernstein certificate, the reduction of $0<a<\tfrac12$ to chords of the
envelope, and finally the atomic verification.

\subsection{Convex envelope of an auxiliary quartic}\label{app:envelope}

Let $a$ and $u$ be arbitrary parameters satisfying
$0<a<\tfrac12$ and
\begin{equation}\label{eq:u-range}
       1-a\leq u\leq\gamma(a)=1+a-2a^2
\end{equation}
(below, $a$ and $u$ will be the means $\E R$ and $\E\gamma(R)$, but this is
not used in the present subsection). Introduce the quartic
\begin{equation}\label{eq:F-def}
 F_{a,u}(r)\eq(3a+2)r^2-2r^3+a(\gamma(r)-u)^2
\end{equation}
and let $\widehat F\eq\co F_{a,u}$ be the greatest convex function on
$[0,1]$ that does not exceed $F_{a,u}$. We need to understand the arrangement
of its intervals of linearity. It turns out that there can be only one
nontrivial such interval: either an \emph{interior chord} with two interior
contact points, or a \emph{right chord} ending at the right endpoint of the
interval (Figure~\ref{fig:convex-envelope}).

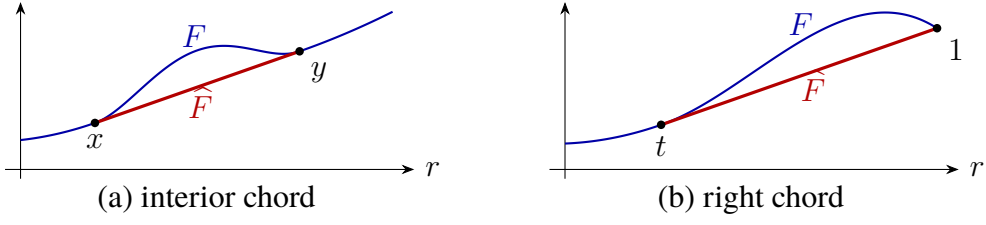
\begin{figure}[ht]
\centering
\begin{tikzpicture}[x=.82cm,y=.82cm,>=Stealth]
\begin{scope}
 \draw[->] (-.25,0)--(6.35,0) node[right] {$r$};
 \draw[->] (0,-.15)--(0,2.7);
 \draw[thick,blue!65!black]
  plot[domain=0:1.2,samples=35]
       (\x,{.75+.35*(\x-1.2)+.10*(\x-1.2)^2});
 \draw[thick,blue!65!black]
  plot[domain=1.2:4.5,samples=70]
       (\x,{.75+.35*(\x-1.2)+.08*(\x-1.2)^2*(\x-4.5)^2});
 \draw[thick,blue!65!black]
  plot[domain=4.5:6,samples=35]
       (\x,{.75+.35*(\x-1.2)+.05*(\x-4.5)^2});
 \draw[very thick,red!70!black] (1.2,.75)--(4.5,1.905);
 \fill (1.2,.75) circle (1.6pt) node[below] {$x$};
 \fill (4.5,1.905) circle (1.6pt) node[below right] {$y$};
 \node[blue!65!black] at (2.8,2.15) {$F$};
 \node[red!70!black] at (2.9,1.08) {$\widehat F$};
 \node at (3,-.48) {(a) interior chord};
\end{scope}
\begin{scope}[xshift=7.2cm]
 \draw[->] (-.25,0)--(6.35,0) node[right] {$r$};
 \draw[->] (0,-.15)--(0,2.7);
 \draw[thick,blue!65!black]
  plot[domain=0:1.55,samples=40]
       (\x,{.72+.35*(\x-1.55)+.10*(\x-1.55)^2});
 \draw[thick,blue!65!black]
  plot[domain=1.55:6,samples=80]
       (\x,{.72+.35*(\x-1.55)+.05*(\x-1.55)^2*(6-\x)});
 \draw[very thick,red!70!black] (1.55,.72)--(6,2.2775);
 \fill (1.55,.72) circle (1.6pt) node[below] {$t$};
 \fill (6,2.2775) circle (1.6pt) node[below right] {$1$};
 \node[blue!65!black] at (3.8,2.35) {$F$};
 \node[red!70!black] at (4.0,1.35) {$\widehat F$};
 \node at (3,-.48) {(b) right chord};
\end{scope}
\end{tikzpicture}
\caption{The two possible nontrivial pieces of the convex envelope of a quartic
(schematically). Outside the chord shown, $\widehat F$ is the same as $F$.}
\label{fig:convex-envelope}
\end{figure}

\begin{lemma}[shape of the convex envelope]\label{lem:conv-envelope}
For $0<a<\tfrac12$ and \eqref{eq:u-range} the set
$\{r:\widehat F(r)<F_{a,u}(r)\}$ is either empty or one
interval. Its closure has one of two types:
\[
          [x,y],\quad 0<x<y<1,
          \qquad\text{or}\qquad [t,1],\quad 0<t<1.
\]
On this interval $\widehat F$ is a chord of $F_{a,u}$, and outside it
$\widehat F=F_{a,u}$.
\end{lemma}

\begin{proof}
\step{Step 1: general structure of the envelope.}
We have
\begin{equation}\label{eq:F-second}
 \frac14F_{a,u}''(r)
 =12ar^2-(6a+3)r+2au+1.
\end{equation}
The right-hand side is a quadratic polynomial with positive leading
coefficient and a positive value at $r=0$. Therefore $F'$ first strictly
increases, may then strictly decrease, and finally strictly increases again,
where the last interval may be absent from $[0,1]$. In other words, the sign of
$F_{a,u}''$ changes at most twice, and the curvature pattern is no more
complicated than \enquote{convex---concave---convex}.

We use an elementary one-dimensional description of the convex envelope. It is
the one-dimensional form of the general formula
\cite[\S~17, Corollary~17.1.5]{Rockafellar1973}: two points $x,y$ suffice in
the convex combination, and the condition $r=\lambda x+(1-\lambda)y$
determines their weights uniquely. Continuity of $F$ on the compact interval
$[0,1]$ guarantees that the minimum is attained. Thus
\[
 \widehat F(r)=\min_{0\leq x\leq r\leq y\leq1}
 \left(\frac{y-r}{y-x}F(x)+\frac{r-x}{y-x}F(y)\right),
\]
where with $x=y=r$ the expression is understood as $F(r)$.
We record two consequences of this formula that will be needed below. Set
\[
 U\eq\{r\in[0,1]:\widehat F(r)<F(r)\}.
\]
The functions $F$ and $\widehat F$ are continuous, so $U$ is relatively open
in $[0,1]$. Every connected component of an open subset of the line is an
interval. We show that $\widehat F$ is affine on every such component.

Fix $r\in U$ and choose a pair $p\leq r\leq q$ attaining the minimum in the
formula for $\widehat F(r)$. The inequality $\widehat F(r)<F(r)$ implies
$p<r<q$, since if $p=r$ or $q=r$, the corresponding convex combination would
equal $F(r)$. Let
\[
 \ell_{p,q}(s)\eq
 \frac{q-s}{q-p}F(p)+\frac{s-p}{q-p}F(q)
\]
be the affine function whose graph contains the chord joining
$(p,F(p))$ and $(q,F(q))$. For $s\in[p,q]$ from the convexity $\widehat F$ and
the inequality $\widehat F\leq F$ imply
\[
 \widehat F(s)
 \leq\frac{q-s}{q-p}\widehat F(p)
     +\frac{s-p}{q-p}\widehat F(q)
 \leq\ell_{p,q}(s).
\]
At $s=r$, the first and last terms in this chain are equal by the choice of
$p,q$, so both inequalities are equalities. Therefore
$\ell_{p,q}-\widehat F$ is nonnegative, concave, and vanishes at the interior
point $r$. Such a function vanishes identically on $[p,q]$:
if $s<r<t$, then from concavity
\[
 0=(\ell_{p,q}-\widehat F)(r)
 \geq\lambda(\ell_{p,q}-\widehat F)(s)
 +(1-\lambda)(\ell_{p,q}-\widehat F)(t)\geq0
\]
for a suitable $\lambda\in(0,1)$, so both terms must be zero. Thus every point
of $U$ has a neighborhood on which $\widehat F$ is affine. On a connected
component these neighborhoods overlap, so their slopes agree and $\widehat F$
is affine on the entire component.

Suppose now that the closure of a component is $[x,y]$, where
$0<x<y<1$. At its ends
\[
 \widehat F(x)=F(x),\qquad \widehat F(y)=F(y),
\]
and on $[x,y]$ the function $\widehat F$ coincides with the line of slope
\[
 m\eq\frac{F(y)-F(x)}{y-x}.
\]
We explain why this line is tangent to the graph of $F$ at both points. At any
interior contact point $c$ where $\widehat F(c)=F(c)$, the nonnegative function
$F-\widehat F$ has a minimum. One-sided derivative inequalities and convexity
of $\widehat F$ give
\[
 F'(c)\leq\widehat F'_-(c)
 \leq\widehat F'_+(c)\leq F'(c).
\]
Thus all three quantities are equal. Applying this at $c=x$ and $c=y$ gives
\[
 F'(x)=m=F'(y)
       =\frac{F(y)-F(x)}{y-x}.
\]
Thus the chord is a \emph{common tangent}: the same line is tangent to the
graph of $F$ at both $x$ and $y$. We now classify all interior and endpoint
chords and thereby prove that $U$ has at most one component.

\pagebreak
\step{Step 2: interior chord.}
For brevity write $F=F_{a,u}$. Suppose that a chord is tangent to the graph at
two interior points $x<y$. Introduce its slope and supporting line:
\[
 m\eq\frac{F(y)-F(x)}{y-x},\qquad
 L_{x,y}(r)\eq F(x)+m(r-x).
\]
The tangency conditions have the form
\[
 F'(x)=m,\qquad F'(y)=m.
\]
We show explicitly how the auxiliary quantity $m$ is eliminated. Set
$G(r)\eq F(r)-L_{x,y}(r)$. Since the line passes through both contact points,
\[
 G(x)=G'(x)=G(y)=G'(y)=0.
\]
Thus $x$ and $y$ are double roots of the quartic $G$.
Expanding the definition of \eqref{eq:F-def} gives
\[
 F(r)=4ar^4-(4a+2)r^3+(4au+2)r^2
      +2a(1-u)r+a(1-u)^2.
\]
Subtracting the linear function $L_{x,y}$ leaves the three highest-degree
coefficients unchanged. Since the leading coefficient is $4a$, we obtain
\[
 G(r)=4a(r-x)^2(r-y)^2.
\]
Write $s=x+y$ and $p=xy$. Then
\[
 (r-x)^2(r-y)^2
 =r^4-2sr^3+(s^2+2p)r^2-2spr+p^2.
\]
Comparing the coefficients of $r^3$ gives
\[
 -8as=-(4a+2),\qquad
 s=\frac{2a+1}{4a}.
\]
Comparing the coefficients of $r^2$ gives
\[
 4a(s^2+2p)=4au+2,\qquad
 p=\frac12\left(u+\frac1{2a}-s^2\right).
\]
Finally,
\[
 (y-x)^2=(x+y)^2-4xy
 =s^2-4p
 =3s^2-2u-\frac1a.
\]
Substituting the value of $s$ gives
\begin{equation}\label{eq:internal-tangent-data}
 x+y=\frac{2a+1}{4a},\qquad
 (y-x)^2=\frac{-32a^2u+12a^2-4a+3}{16a^2}.
\end{equation}
We will refer to this elementary algebraic elimination of the auxiliary
variable $m$ as \emph{elimination of the slope}. Formulas
\eqref{eq:internal-tangent-data} determine the unordered pair $\{x,y\}$
uniquely: if $d=y-x>0$, then
\[
 x=\frac{s-d}{2},\qquad y=\frac{s+d}{2}.
\]
Hence, if an interior pair of tangency points exists, it is unique. The same
calculation also gives the exact factorization
\begin{equation}\label{eq:internal-chord-factor}
 F(r)-L_{x,y}(r)=4a(r-x)^2(r-y)^2.
\end{equation}
Its right-hand side is nonnegative, so the interior chord does lie below the
graph. Differentiating \eqref{eq:internal-chord-factor} twice at $x,y$ gives
\[
 F''(x)=F''(y)=8a(y-x)^2>0.
\]
The axis of symmetry of the quadratic polynomial \eqref{eq:F-second} is
\[
 \frac{2a+1}{8a}=\frac{x+y}{2}.
\]
Since its leading coefficient is positive and its values at $x,y$ are positive,
$F''\geq0$ on $[0,x]$ and $[y,1]$. Thus $F$ is convex on these two intervals.

Define $\Phi$ to equal $F$ outside $[x,y]$ and $L_{x,y}$ on $[x,y]$.
Its slopes agree at the junctions because $F'(x)=m=F'(y)$; hence $\Phi$ is
convex. Formula~\eqref{eq:internal-chord-factor} gives $\Phi\leq F$. Finally,
if $h\leq F$ is any convex function, then for $r\in[x,y]$
\[
 h(r)\leq\frac{y-r}{y-x}h(x)+\frac{r-x}{y-x}h(y)
 \leq\frac{y-r}{y-x}F(x)+\frac{r-x}{y-x}F(y)
 =L_{x,y}(r).
\]
Outside $[x,y]$ we also have $h\leq F=\Phi$. Thus $h\leq\Phi$ throughout
$[0,1]$, and hence $\Phi=\widehat F$. Consequently, if an interior chord
exists, the set $U$ has no second component.

\step{Step 3: right chord.}
Let a chord join $(t,F(t))$ and $(1,F(1))$, where $t<1$. Write
\[
 m_{t,1}\eq\frac{F(1)-F(t)}{1-t},\qquad
 L_{t,1}(r)\eq F(t)+m_{t,1}(r-t).
\]
The interior endpoint $t$ is a point of tangency, so $F'(t)=m_{t,1}$.
Direct substitution of the polynomial $F$ gives
\[
 F'(t)-m_{t,1}
 =4(t-1)(3at^2-t+au).
\]
Since $t<1$, the tangency condition is equivalent to the equation
\begin{equation}\label{eq:right-tangent-data}
                 3at^2-t+au=0.
\end{equation}
Since $a>0$ and $u\geq1-a>0$, the value $t=0$ does not satisfy
\eqref{eq:right-tangent-data}; hence here $0<t<1$. Substituting
\eqref{eq:right-tangent-data} into $F-L_{t,1}$ and collecting factors gives
the exact factorization
\begin{equation}\label{eq:right-chord-factor}
 F(r)-L_{t,1}(r)=2(r-1)(r-t)^2(2ar+4at-1).
\end{equation}
For $r\in[t,1]$, the first two factors on the right satisfy
$r-1\leq0$ and $(r-t)^2\geq0$, while the last factor is increasing in $r$.
Therefore the entire difference is nonnegative on $[t,1]$ if and only if its
last factor is nonpositive at $r=1$, that is,
\[
 2a+4at-1\leq0.
\]
If this condition fails, the difference is negative immediately to the left of
$1$, and the chord is not a minorant.

We also prove uniqueness of the admissible root of
\eqref{eq:right-tangent-data}. If $a<\tfrac16$, the sum of the two roots of
the quadratic equation is $1/(3a)>2$, so they cannot both lie in $[0,1]$.
If $a\geq\tfrac16$, the preceding chord condition gives
\[
 t\leq\frac{1-2a}{4a}\leq\frac1{6a},
\]
and the polynomial $3at^2-t+au$ is strictly decreasing to the left of its
vertex $1/(6a)$. Thus in either case there is at most one admissible root.

\step{Step 4: exclusion of chords issuing from zero.}
It remains to consider a chord joining $(0,F(0))$ to an interior point
$(t,F(t))$, where $t\in(0,1)$. Its slope and line are
\[
 m_{0,t}\eq\frac{F(t)-F(0)}{t},\qquad
 L_{0,t}(r)\eq F(0)+m_{0,t}r.
\]
Tangency at $t$ requires $F'(t)=m_{0,t}$. Direct calculation gives
\[
 F'(t)-m_{0,t}
 =2t\bigl(6at^2-(4a+2)t+2au+1\bigr).
\]
Since $t>0$, the tangency condition is equivalent to
\[
 6at^2-(4a+2)t+2au+1=0.
\]
Substituting this identity, the difference between $F(r)$ and the chord
factorizes as
\begin{equation}\label{eq:left-chord-factor}
 F(r)-L_{0,t}(r)
 =2r(r-t)^2(2ar+4at-2a-1).
\end{equation}
At $r=0$, the last factor is at most
\[
 4a-2a-1=2a-1<0.
\]
It therefore remains negative for all sufficiently small $r>0$. Choosing also
$r<t$, formula~\eqref{eq:left-chord-factor} yields
$F(r)-L_{0,t}(r)<0$. Thus the chord lies above $F$ in a right neighborhood of
zero and is not a minorant.

It remains to consider the chord between the endpoints $0$ and $1$. Its slope
exceeds the derivative of $F$ at zero:
\[
 m_{0,1}-F'(0)
 =\frac{F(1)-F(0)}{1-0}-F'(0)=4au>0.
\]
Thus $L_{0,1}-F$ vanishes at zero and has positive right derivative there.
Consequently, $L_{0,1}(r)>F(r)$ for all sufficiently small $r>0$, so this
chord is not a minorant either.

The classification is complete. If a component of $U$ has two interior
endpoints, it is the unique interior chord from Step~2. If no interior pair
exists, every nonempty component must be adjacent to $0$ or to $1$. The first
possibility was excluded in this step, while in the second there is at most one
right chord by Step~3. Hence $U$ is either empty or consists of exactly one
interval of the type specified in the lemma. This proves the lemma.
\end{proof}

\subsection{Range \texorpdfstring{$a\geq\tfrac12$}{a >= 1/2}}

\begin{proof}[Proof of Lemma~\ref{lem:moment}]
Recall the definitions
\[
 a\eq\E R,\qquad b\eq\E\bigl[R^2(1-R)\bigr].
\]
Since $a>0$, multiplying both sides of \eqref{eq:moment-main} by $12a$
rewrites the assertion equivalently as
\begin{equation}\label{eq:moment-cleared}
 a\left(3\Var(\varepsilon R)+\Var(\gamma(R))\right)
 \geq2(a^2-b).
\end{equation}

First suppose that $a\geq\tfrac12$. Write
\[
 e\eq\E(\varepsilon R),\quad u\eq\E\gamma(R),\quad
 c\eq\E\bigl[R^3(1-R)\bigr].
\]
The identity
\begin{equation}\label{eq:gamma-identity}
 3r^2+\gamma(r)^2+4r^3(1-r)=1+2r
\end{equation}
may be averaged at $r=R$. By the definitions of $a$ and $c$, this gives
\[
 3\E R^2+\E\!\left[\gamma(R)^2\right]=1+2a-4c.
\]
Moreover, since $\varepsilon^2=1$,
\[
 \Var(\varepsilon R)=\E R^2-e^2,\qquad
 \Var(\gamma(R))=\E\!\left[\gamma(R)^2\right]-u^2.
\]
Therefore,
\[
 3\Var(\varepsilon R)+\Var(\gamma(R))
 =1+2a-4c-3e^2-u^2.
\]
Substituting this identity into \eqref{eq:moment-cleared} gives the following
equivalent transformations:
\begin{align*}
 a\bigl(1+2a-4c-3e^2-u^2\bigr)&\geq2(a^2-b),\\
 3ae^2+au^2&\leq a+2b-4ac.
\end{align*}
Finally, division by $a>0$ shows that \eqref{eq:moment-cleared} is equivalent
to
\begin{equation}\label{eq:mean-vector-bound}
 3e^2+u^2\leq 1+\frac{2b}{a}-4c.
\end{equation}

Consider a random vector
\[
 W\eq\left(\sqrt3\,\varepsilon(R-\tfrac12),\;
          \gamma(R)+R-a\right).
\]
Since $\E\varepsilon=0$, the definitions of $a,e,u$ give
\[
 \E W
 =\left(\sqrt3\bigl(\E(\varepsilon R)-\tfrac12\E\varepsilon\bigr),\;
         \E\gamma(R)+\E R-a\right)
 =(\sqrt3e,u).
\]
Apply Jensen's inequality to the convex function
$\Phi(w)=\lVert w\rVert^2$ on $\mathbb R^2$:
\[
 3e^2+u^2=\lVert\E W\rVert^2\leq\E\lVert W\rVert^2.
\]
Since $\varepsilon^2=1$,
\[
 \lVert W\rVert^2
 =3(R-\tfrac12)^2+\bigl(\gamma(R)+R-a\bigr)^2.
\]
It is therefore enough to prove, for every $r\in[0,1]$, the pointwise
inequality
\begin{equation}\label{eq:pointwise-large-a}
 3(r-\tfrac12)^2+(\gamma(r)+r-a)^2
 \leq 1+\frac{2r^2(1-r)}a-4r^3(1-r).
\end{equation}
Indeed, substituting $r=R$ in \eqref{eq:pointwise-large-a} and averaging gives,
by the definitions of $b,c$,
\[
 \E\lVert W\rVert^2
 \leq1+\frac2a\E\bigl[R^2(1-R)\bigr]
        -4\E\bigl[R^3(1-R)\bigr]
 =1+\frac{2b}{a}-4c.
\]
Together with Jensen's inequality this gives
\[
 3e^2+u^2
 =\lVert\E W\rVert^2
 \leq\E\lVert W\rVert^2
 \leq1+\frac{2b}{a}-4c,
\]
which is exactly \eqref{eq:mean-vector-bound}.

Let $t\eq2a-1\in[0,1]$. The difference between the right- and left-hand sides
of \eqref{eq:pointwise-large-a} is
\[
 \frac{\mathcal B(r,t)}{4(t+1)},
\]
where
\begin{align*}
\mathcal B(r,t)={}&16r^3t-8r^2t^2-28r^2t-4r^2+8rt^2+12rt+4r\\
         &{}-t^3+t^2+2t.
\end{align*}
Thus the entire case $a\geq\tfrac12$ reduces to proving that the polynomial
$\mathcal B(r,t)$ is nonnegative on the square $[0,1]^2$: the denominator
$4(t+1)$ is positive, and as $a$ varies, $t=2a-1$ ranges over $[0,1]$.
The total degree of $\mathcal B$ is $4$, while its degrees in $r$ and $t$
separately are at most $3$; we abbreviate this as coordinate degree $(3,3)$.
We verify nonnegativity using the Bernstein basis.

\medskip
\noindent\textbf{One-dimensional Bernstein basis.}
For a fixed $n\geq0$ we set
\[
 B_{k,n}(x)\eq\binom nk x^k(1-x)^{n-k},
 \qquad k=0,\dots,n.
\]
These polynomials form a basis of the space of polynomials of degree at most
$n$. They are nonnegative on $[0,1]$, and the binomial theorem gives
\[
 \sum_{k=0}^n B_{k,n}(x)=1.
\]
The connection with the usual monomial basis is given by the triangular formula
\[
 x^j=\sum_{k=j}^n
 \frac{\binom{k}{j}}{\binom{n}{j}}\,B_{k,n}(x),
 \qquad 0\leq j\leq n.
\]
Thus every polynomial $p$ of degree at most $n$ has a unique representation
\[
 p(x)=\sum_{k=0}^n\beta_k B_{k,n}(x).
\]
For fixed $x\in[0,1]$, this value is a convex combination of the numbers
$\beta_k$; therefore
\[
 \min_k\beta_k\leq p(x)\leq\max_k\beta_k.
\]
In particular, nonnegativity of all $\beta_k$ is sufficient for nonnegativity
of $p$ on $[0,1]$.

\medskip
\noindent\textbf{Tensor basis.}
For $d$ variables $z=(z_1,\dots,z_d)$ and coordinate degree
$D=(D_1,\dots,D_d)$, set
\[
 B_{k,D}(z)\eq\prod_{i=1}^d B_{k_i,D_i}(z_i),
 \qquad 0\leq k_i\leq D_i.
\]
These tensor products form a basis of the space of polynomials satisfying
$\deg_{z_i}\Pi\leq D_i$. If
\[
 \Pi(z)=\sum_\alpha c_\alpha z^\alpha
       =\sum_{0\leq k\leq D}\beta_k B_{k,D}(z),
\]
then applying the preceding one-dimensional formula coordinatewise gives
\begin{equation}\label{eq:bernstein-coeff-formula}
 \beta_k=\sum_{\alpha\leq k}c_\alpha
       \prod_{i=1}^d
       \frac{\binom{k_i}{\alpha_i}}{\binom{D_i}{\alpha_i}},
\end{equation}
where inequalities between multi-indices are interpreted coordinatewise. The
tensor-product basis polynomials are nonnegative on $[0,1]^d$ and sum to one.
We therefore call an expansion with all coefficients nonnegative a
\emph{certificate of nonnegativity in the Bernstein basis}. If the original and
resulting coefficients are rational and are computed exactly, this is an
\emph{exact rational certificate}. Standard properties of the basis and their
proofs may be found in~\cite[Section~2]{MichelucciFoufouKubicki2012}.

In our case, \enquote{the tensor basis of degree $(3,3)$} means the expansion
\[
 \mathcal B(r,t)
 =\sum_{k_r=0}^3\sum_{k_t=0}^3
   \beta_{k_r,k_t}B_{k_r,3}(r)B_{k_t,3}(t).
\]
In the following table, the row indexed by $k_r$ and the column indexed by
$k_t$ contain the coefficient $\beta_{k_r,k_t}$:
\[
\begin{array}{c|cccc}
 & k_t=0 & k_t=1 & k_t=2 & k_t=3\\
\hline
k_r=0 & 0&2/3&5/3&2\\
k_r=1 & 4/3&10/3&59/9&10\\
k_r=2 & 4/3&14/9&3&14/3\\
k_r=3 & 0&2/3&5/3&2
\end{array}.
\]
Thus the table is an exact rational certificate: all of its entries are
nonnegative, and hence $\mathcal B\geq0$ on $[0,1]^2$.
This proves the range $a\geq\tfrac12$.

\newpage
\subsection{Range \texorpdfstring{$0<a<\tfrac12$}{0 < a < 1/2}: reduction to chords of the envelope}
\label{app:chord-reduction}

\emph{Conditional laws and averaging.}
Let $P_+$ and $P_-$ denote the conditional laws of $R$ given
$\varepsilon=1$ and $\varepsilon=-1$. For any two probability measures
$P,Q$ on $[0,1]$, introduce averaging with respect to their half-sum:
\begin{equation}\label{eq:bar-expectation}
 \overline{\E}_{P,Q}\!\left[\phi(R)\right]
 \eq\frac12\left(\int\phi(r)P(\dd r)+\int\phi(r)Q(\dd r)\right).
\end{equation}
For the original pair we abbreviate
$\overline{\E}=\overline{\E}_{P_+,P_-}$. Since the two values of
$\varepsilon$ are equally likely, the law of total expectation gives
\begin{align*}
 \E\phi(R)
 &=\overline{\E}[\phi(R)],\\
 \E\bigl[\varepsilon\phi(R)\bigr]
 &=\frac12\left(\int\phi(r)P_+(\dd r)
                 -\int\phi(r)P_-(\dd r)\right).
\end{align*}
Thus, for functions of $R$ alone, $\E$ and $\overline{\E}$ have the same
numerical value. The indexed notation \eqref{eq:bar-expectation} allows us to
use the same formulas after replacing the pair of measures.

Set
\[
 p\eq\int r\,P_+(\dd r),\qquad
 q\eq\int r\,P_-(\dd r).
\]
The parameters $a,e,u$ introduced above are expressed through the conditional
laws as
\[
 a=\overline{\E}[R]=\frac{p+q}{2},\qquad
 e=\E(\varepsilon R)=\frac{p-q}{2},\qquad
 u=\overline{\E}[\gamma(R)].
\]
Therefore,
\[
 p=a+e,\qquad q=a-e.
\]
If necessary, replacing $\varepsilon$ by $-\varepsilon$ interchanges $P_+$
and $P_-$ and replaces $e$ by $-e$. We may therefore assume $p\geq q$, that
is, $e\geq0$.

By Jensen’s inequality for the concave function $\gamma$ and by the estimate
$\gamma(r)\geq1-r$ we have
\[
 u\leq\gamma(\overline{\E}[R])=\gamma(a),
 \qquad
 u\geq\overline{\E}[1-R]=1-a.
\]
Thus the parameters $a,u$ satisfy \eqref{eq:u-range}, so the results of
Subsection~\ref{app:envelope} apply to the quartic $F_{a,u}$ from
\eqref{eq:F-def}.

\emph{Polynomial form of the deficit.}
Define $\mathcal N(P_+,P_-)$ as the difference between the left- and
right-hand sides of \eqref{eq:moment-cleared}. Using
\[
 \Var(\varepsilon R)=\overline{\E}\!\left[R^2\right]-e^2,\qquad
 \Var(\gamma(R))
 =\overline{\E}\!\left[\gamma(R)^2\right]-u^2,\qquad
 b=\overline{\E}\bigl[R^2(1-R)\bigr],
\]
and then $\overline{\E}[\gamma(R)]=u$, we obtain
\begin{equation}\label{eq:N-as-F}
\begin{aligned}
 \mathcal N(P_+,P_-)
 &\eq
 a\left(3\Var(\varepsilon R)+\Var(\gamma(R))\right)
   -2(a^2-b)\\
 &=a\left(
      3\bigl(\overline{\E}[R^2]-e^2\bigr)
      +\overline{\E}[\gamma(R)^2]-u^2
    \right)
   -2a^2+2\overline{\E}[R^2-R^3]\\
 &=\overline{\E}\!\left[
      (3a+2)R^2-2R^3+a\gamma(R)^2
    \right]
   -2a^2-3ae^2-au^2\\
 &=\overline{\E}\!\left[
      (3a+2)R^2-2R^3+a\bigl(\gamma(R)-u\bigr)^2
    \right]
   -2a^2-3ae^2\\
 &=\frac12\left(\int F_{a,u}(r)P_+(\dd r)
                 +\int F_{a,u}(r)P_-(\dd r)\right)
   -2a^2-3ae^2.
\end{aligned}
\end{equation}

\emph{Passage to the convex envelope.}
By the definition of the convex envelope and Lemma~\ref{lem:conv-envelope},
for every $r\in[0,1]$ there is a probability measure $Q_r$ with mean $r$ such
that
\begin{equation}\label{eq:envelope-realize}
 \int F_{a,u}(s)Q_r(\dd s)=\widehat F(r).
\end{equation}
Indeed, if $\widehat F(r)=F_{a,u}(r)$, take
$Q_r=\delta_r$. If $r$ lies on the linear chord $[x,y]$, then
\[
 Q_r=\frac{y-r}{y-x}\,\delta_x+\frac{r-x}{y-x}\,\delta_y
\]
is the unique mixture of $\delta_x,\delta_y$ with mean $r$.

Since $\widehat F\leq F_{a,u}$ and $\widehat F$ is convex, Jensen's inequality
gives, for each of the measures $P_\pm$,
\[
 \int F_{a,u}(r)P_\pm(\dd r)
 \geq\int\widehat F(r)P_\pm(\dd r)
 \geq\widehat F\left(\int r\,P_\pm(\dd r)\right).
\]
Substituting these estimates into \eqref{eq:N-as-F} gives
\begin{equation}\label{eq:envelope-lower}
 \mathcal N(P_+,P_-)
 \geq \frac12\bigl(\widehat F(p)+\widehat F(q)\bigr)-2a^2-3ae^2.
\end{equation}

\emph{Replacement of the measures.}
Take $\widetilde P_+=Q_p$ and $\widetilde P_-=Q_q$ from
\eqref{eq:envelope-realize}. Their means are $p$ and $q$, respectively, so
$a=(p+q)/2$ and $e=(p-q)/2$ do not change. The mean of the nonlinear
function $\gamma$, however, may change; denote it by
\[
 \widetilde u
 \eq\overline{\E}_{\widetilde P_+,\widetilde P_-}
       \!\left[\gamma(R)\right]
 =\frac12\left(\int\gamma(r)\widetilde P_+(\dd r)
                +\int\gamma(r)\widetilde P_-(\dd r)\right).
\]
This is the source of the correction in the following identity. The measures
$Q_p,Q_q$ in \eqref{eq:envelope-realize} were constructed for the original
polynomial $F_{a,u}$, and therefore
\[
 \frac12\bigl(\widehat F(p)+\widehat F(q)\bigr)
 =\frac12\sum_{\sigma\in\{+,-\}}
   \int F_{a,u}(r)\widetilde P_\sigma(\dd r).
\]
On the other hand, formula~\eqref{eq:N-as-F} applied to the new pair of
measures contains its own parameter $\widetilde u$:
\[
 \mathcal N(\widetilde P_+,\widetilde P_-)
 =\frac12\sum_{\sigma\in\{+,-\}}
   \int F_{a,\widetilde u}(r)\widetilde P_\sigma(\dd r)
  -2a^2-3ae^2.
\]
The difference between the two averaged polynomials is
\[
\begin{aligned}
 &\frac12\sum_{\sigma\in\{+,-\}}
   \int\bigl(F_{a,u}(r)-F_{a,\widetilde u}(r)\bigr)
        \widetilde P_\sigma(\dd r)\\
 &\quad
 =\frac a2\sum_{\sigma\in\{+,-\}}
   \int\left((\gamma(r)-u)^2
             -(\gamma(r)-\widetilde u)^2\right)
        \widetilde P_\sigma(\dd r)\\
 &\quad
 =a\left(2(\widetilde u-u)\widetilde u
          +u^2-\widetilde u^2\right)
 =a(\widetilde u-u)^2.
\end{aligned}
\]
The penultimate equality uses the definition of $\widetilde u$ and the fact
that both measures $\widetilde P_\pm$ have total mass $1$.
Therefore,
\begin{equation}\label{eq:replacement-identity}
 \frac12\bigl(\widehat F(p)+\widehat F(q)\bigr)-2a^2-3ae^2
 =\mathcal N(\widetilde P_+,\widetilde P_-)
  +a(\widetilde u-u)^2.
\end{equation}
The correction vanishes only in the special case $\widetilde u=u$; preserving
the means $p,q$ alone does not ensure this equality. Since $a>0$, the last
term is nonnegative. Combining \eqref{eq:envelope-lower} and
\eqref{eq:replacement-identity}, we obtain
\[
 \mathcal N(P_+,P_-)
 \geq\mathcal N(\widetilde P_+,\widetilde P_-)
     +a(\widetilde u-u)^2
 \geq\mathcal N(\widetilde P_+,\widetilde P_-).
\]
It is therefore enough to prove $\mathcal N(Q_p,Q_q)\geq0$ for all pairs of
canonical measures used to realize \eqref{eq:envelope-realize}. Each is either
a point mass at a contact point where $\widehat F=F_{a,u}$ or a two-point
mixture at the endpoints of the unique chord of the envelope.

\subsection{Final atomic verification of the moment lemma}\label{app:atomic-check}

In Subsection~\ref{app:chord-reduction}, for each $r$ we chose a canonical
measure $Q_r$ attaining the value $\widehat F(r)$ in
\eqref{eq:envelope-realize}. By Lemma~\ref{lem:conv-envelope}, it has one of
two forms:
$Q_r=\delta_r$ at the point of coincidence $\widehat F(r)=F_{a,u}(r)$ or
\[
 Q_r=(1-\lambda)\delta_\xi+\lambda\delta_\eta,
 \qquad
 r=(1-\lambda)\xi+\lambda\eta,
\]
where $[\xi,\eta]$ is the unique nontrivial chord of the envelope. Thus, after
the reduction in Subsection~\ref{app:chord-reduction}, only point masses and
two-point mixtures with the same pair of chord endpoints remain.

For a pair of such measures, denote them by $P,Q$ and, retaining the notation
of A.3, set
\begin{equation}\label{eq:atomic-notation}
\begin{aligned}
 p&\eq\int r\,P(\dd r),&
 q&\eq\int r\,Q(\dd r),&
 a&\eq\frac{p+q}{2},&
 e&\eq\frac{p-q}{2}.
\end{aligned}
\end{equation}
In this subsection, abbreviate $\overline{\E}=\overline{\E}_{P,Q}$, where
$\overline{\E}_{P,Q}$ is defined in \eqref{eq:bar-expectation}. The difference
between the left- and right-hand sides of \eqref{eq:moment-cleared} for the
current pair is
\begin{equation}\label{eq:atomic-N}
 \mathcal N(P,Q)\eq
 a\left(3\bigl(\overline\E[R^2]-e^2\bigr)
 +\overline\E[\gamma(R)^2]
 -\bigl(\overline\E[\gamma(R)]\bigr)^2\right)
 -2\left(a^2-\overline\E[R^2(1-R)]\right).
\end{equation}
This is the same quantity $\mathcal N$ as in A.3, now written for an arbitrary
current pair $P,Q$.

\emph{Both measures are point masses.}
If $P=\delta_p$ and $Q=\delta_q$, direct substitution into
\eqref{eq:atomic-N} gives
\begin{equation}\label{eq:dirac-gap}
 \mathcal N(\delta_p,\delta_q)
 =a^3+e^2(16a^3-8a^2-5a+2).
\end{equation}
Here $|e|\leq a$ because $p,q\geq0$. If the coefficient of $e^2$ is
nonnegative, then the right-hand side of \eqref{eq:dirac-gap} is nonnegative.
If that coefficient is negative, $e^2\leq a^2$ gives
\[
 \mathcal N(\delta_p,\delta_q)
 \geq a^3+a^2(16a^3-8a^2-5a+2)
 =2a^2(2a-1)^2(2a+1)\geq0.
\]
Thus the case of two point masses is completely settled. This covers both the
absence of a nontrivial chord and the case in which both selected points lie
outside its interior.

\emph{At least one measure has two points.}
By Lemma~\ref{lem:conv-envelope}, the unique chord is either $[t,1]$ or
$[x,y]$ with $0<x<y<1$. In the first case, a measure on the chord is a
mixture of $\delta_t,\delta_1$, and a point mass outside the chord can lie only
in $[0,t]$. In the second case, a measure on the chord is a mixture of
$\delta_x,\delta_y$, and a point mass outside the chord lies either in $[0,x]$
or in $[y,1]$. If both measures have two points, they use the endpoints of the
same chord.

These possibilities give exactly the five rows in the following table.
Interchanging $P$ and $Q$ creates no new cases because
$\mathcal N(P,Q)=\mathcal N(Q,P)$. The parameters
$t,x,\tau,\lambda,\mu,s$ independently range over $[0,1]$, and
\[
 y=x+(1-x)\tau.
\]
The parametrizations $ts$, $xs$, and $y+(1-y)s$ describe points in $[0,t]$,
$[0,x]$, and $[y,1]$, respectively. In each of the five cases we deliberately
verify the inequality on the entire closed parameter cube $[0,1]^k$, including
its boundary and the corresponding degenerate cases; this only enlarges the set
that must be checked. Moreover, we do not impose the tangency equations relating
the endpoints of the actual chord to the original parameters $a,u$ of
$F_{a,u}$. Nonnegativity on this larger family is more than sufficient for the
canonical measures in \eqref{eq:envelope-realize}.

\begin{center}
\small
\begin{tabular}{@{}>{\raggedright\arraybackslash}p{31mm}
                    >{\raggedright\arraybackslash}p{41mm}
                    >{\raggedright\arraybackslash}p{41mm}
                    >{\raggedright\arraybackslash}p{32mm}@{}}
\toprule
case & $P$ & $Q$ & variables: powers\\
\midrule
(I) both measures on the right chord
& $(1-\lambda)\delta_t+\lambda\delta_1$
& $(1-\mu)\delta_t+\mu\delta_1$
& $(t,\lambda,\mu):(5,3,3)$\\[1mm]
(II) right chord and point on the left
& $(1-\lambda)\delta_t+\lambda\delta_1$
& $\delta_{ts}$
& $(t,\lambda,s):(7,5,7)$\\[1mm]
(III) both measures on the inner chord
& $(1-\lambda)\delta_x+\lambda\delta_y$
& $(1-\mu)\delta_x+\mu\delta_y$
& $(x,\tau,\lambda,\mu):(6,6,4,4)$\\[1mm]
(IV) internal chord and point on the left
& $(1-\lambda)\delta_x+\lambda\delta_y$
& $\delta_{xs}$
& $(x,\tau,\lambda,s):(7,7,5,7)$\\[1mm]
(V) internal chord and point to the right
& $(1-\lambda)\delta_x+\lambda\delta_y$
& $\delta_{y+(1-y)s}$
& $(x,\tau,\lambda,s):(6,6,4,6)$\\
\bottomrule
\end{tabular}
\end{center}

\emph{Exact polynomial test.}
For each row, substitute the indicated atomic measures into
\eqref{eq:atomic-N} and denote the resulting polynomial in the parameters by
$\Pi_{\mathrm I},\ldots,\Pi_{\mathrm V}$. The required inequality in that row
is equivalent to nonnegativity of the corresponding $\Pi_j$ on the unit cube.

The last column of the first table specifies the coordinate degree of the
tensor-product Bernstein basis. For some variables it exceeds the natural
degree of the polynomial; this is standard degree elevation, used to obtain
nonnegative coefficients.

\medskip
\noindent\emph{How degree elevation works.}
The polynomial itself does not change; only its coordinates in the
Bernstein basis. From the equality $1=(1-x)+x$ it immediately follows
that
\begin{equation}\label{eq:bernstein-degree-elevation}
 B_{k,n}(x)
 =\frac{n+1-k}{n+1}B_{k,n+1}(x)
  +\frac{k+1}{n+1}B_{k+1,n+1}(x).
\end{equation}
Therefore, if
\[
 p(x)=\sum_{k=0}^{n}b_kB_{k,n}(x)
     =\sum_{k=0}^{n+1}c_kB_{k,n+1}(x),
\]
then the new coefficients have the form
\[
 c_0=b_0,\qquad
 c_k=\frac{k}{n+1}b_{k-1}
     +\left(1-\frac{k}{n+1}\right)b_k
       \quad(1\leq k\leq n),\qquad
 c_{n+1}=b_n.
\]
Thus each new interior coefficient is a convex combination of two adjacent
old coefficients.

Consider an example unrelated to the present problem:
\[
 p(x)=\left(x-\frac12\right)^2+\frac1{10}
     =x^2-x+\frac7{20}.
\]
In the natural degree basis $2$ it is written as
\[
 p(x)=\frac7{20}B_{0,2}(x)
      -\frac3{20}B_{1,2}(x)
      +\frac7{20}B_{2,2}(x).
\]
Although $p(x)\geq1/10$ on $[0,1]$, the middle coefficient in this expansion
is negative, so the expansion does not certify nonnegativity. Elevate the basis
degree to $3$. Formula
\eqref{eq:bernstein-degree-elevation} gives
\[
\begin{aligned}
 c_0&=\frac7{20},\\
 c_1&=\frac13\frac7{20}
       +\frac23\left(-\frac3{20}\right)=\frac1{60},\\
 c_2&=\frac23\left(-\frac3{20}\right)
       +\frac13\frac7{20}=\frac1{60},\\
 c_3&=\frac7{20}.
\end{aligned}
\]
Therefore, the same polynomial has the representation
\[
 p(x)=\frac7{20}B_{0,3}(x)
      +\frac1{60}B_{1,3}(x)
      +\frac1{60}B_{2,3}(x)
      +\frac7{20}B_{3,3}(x).
\]
Now all coefficients are positive, and nonnegativity of $p$ follows directly
from nonnegativity of the basis polynomials. In a tensor-product basis, this
transformation is performed separately in each variable.

This is how the degrees in the last column of the first table are to be
understood. Expanding the expressions and applying
\eqref{eq:bernstein-coeff-formula} after the selected degree elevations gives
the following entirely rational verification.

\begin{center}
\small
\begin{tabular}{@{}lrrrr@{}}
\toprule
case & total & $>0$ & $=0$ & smallest $>0$\\
\midrule
(I) & 96 & 87 & 9 & $1/10$\\
(II) & 384 & 324 & 60 & $1/882$\\
(III) & 1225 & 1073 & 152 & $1/240$\\
(IV) & 3072 & 2642 & 430 & $1/1715$\\
(V) & 1715 & 1636 & 79 & $1/270$\\
\bottomrule
\end{tabular}
\end{center}

There are no negative coefficients, and the largest coefficient in all five
rows is $1$. Since the Bernstein basis polynomials are nonnegative on the unit
cube, it follows that $\Pi_j\geq0$ in all five cases. The complete rational
coefficient lists and two independent implementations of the exact
verification are provided in the supplementary materials; the implementations
produce bit-for-bit identical results.

Together with the separately analyzed pair of point masses, this proves
$\mathcal N(P,Q)\geq0$ for each pair of measures that can occur in
\eqref{eq:envelope-realize}. Consequently,
\eqref{eq:envelope-lower}--\eqref{eq:replacement-identity} give
\[
             \mathcal N(P_+,P_-)\geq0
\]
for arbitrary $P_+,P_-$. Together with the range $a\geq\tfrac12$ already
treated, this proves \eqref{eq:moment-cleared}. Dividing both sides by
$12a>0$ gives
\[
 \frac14\Var(\varepsilon R)+\frac1{12}\Var(\gamma(R))
 \geq\frac{a^2-b}{6a},
\]
which is exactly \eqref{eq:moment-main}. This proves
Lemma~\ref{lem:moment}.
\end{proof}

\begin{remark}[why the two ranges are different]
In the proof for $a\geq\tfrac12$, all transformations preceding the pointwise
inequality \eqref{eq:pointwise-large-a} use only $a>0$. The condition
$a\geq\tfrac12$ becomes essential in the substitution
\[
 t_{\mathrm B}=2a-1
\]
(this parameter is denoted simply by $t$ in the proof). Since
$a=\E R\leq1$, precisely in this range $t_{\mathrm B}\in[0,1]$, and the
difference between the sides of \eqref{eq:pointwise-large-a} becomes the
polynomial $\mathcal B(r,t_{\mathrm B})$ on the unit square. A single
Bernstein certificate therefore suffices.

For $0<a<\tfrac12$, this pointwise strategy is not merely unsupported by the
certificate; it is false. Already at $r=0$, the difference between the right-
and left-hand sides of \eqref{eq:pointwise-large-a} is
\[
 1-\left(\frac34+(1-a)^2\right)
 =\frac{(2a-1)(3-2a)}4<0.
\]
Thus the desired inequality emerges only after averaging: negative
contributions from some values of $R$ must be compensated by others, subject
to the prescribed means. The laws $P_+,P_-$ must therefore be retained rather
than replacing the problem by a single pointwise estimate.

Formally, the condition $0<a<\tfrac12$ enters the main proof when
Lemma~\ref{lem:conv-envelope} is applied. The key sign occurs in Step~4 of
that lemma: in factorization~\eqref{eq:left-chord-factor}, the last factor at
zero is bounded above by
\[
 2a-1<0.
\]
This excludes chords issuing from zero and, together with Steps~2--3, leaves at
most one interior or right chord. The admissibility condition for the right
chord,
\[
 2a+4at_{\mathrm{ch}}-1\leq0,\qquad t_{\mathrm{ch}}>0,
\]
where $t_{\mathrm{ch}}$ is its left endpoint, in turn forces
$a<\tfrac12$: this geometry of the envelope can occur only in the small-$a$
range.

The convex envelope is the exact device for handling averaging at fixed mean:
the minimum of the integral of $F_{a,u}$ is attained by a point mass or by a
mixture of chord endpoints. Thus the infinite-dimensional problem in
$P_+,P_-$ is reduced first to the geometry of the envelope and then to five
atomic families. After this reduction, the calculations in the present
subsection no longer use $a<\tfrac12$: the pair of point masses and the five
certificates are verified even on enlarged parameter cubes without the
tangency equations.

The difference in complexity is therefore structural. When
$a\geq\tfrac12$, a single pointwise majorant exists and the problem reduces to
one polynomial in two variables. For $0<a<\tfrac12$, such a majorant is
impossible, so one must use compensation between points, the convex envelope,
its chord classification, and five separate certificates. The same threshold
$2a-1=0$ separates the two mechanisms; the boundary value $a=\tfrac12$
belongs to the first case and corresponds to $t_{\mathrm B}=0$.
\end{remark}

\section{Supplementary proofs and materials}\label{app:supplement}

\subsection{Countability of maximal support segments}
\label{app:countable-support-segments}

The following fact is standard and is not new. Ewald, Larman, and Rogers
explicitly note in~\cite[p.~1]{EwaldLarmanRogers1970} that the maximal line
segments on the boundary of a planar convex region form an at most countable
family and that this is easy to prove. For completeness, we give the short
standard proof.

\begin{lemma}\label{lem:countable-support-segments}
A planar convex body has at most countably many maximal nondegenerate support
faces. Consequently, the set of directions $u\in S^1$ for which the support face
\[
 \mathcal F_K(u)\eq
 \{z\in K:\langle z,u\rangle
   =\max_{w\in K}\langle w,u\rangle\}
\]
is not a singleton is at most countable.
\end{lemma}

\begin{proof}
Every nonsingleton support face of a planar convex body is a closed segment of
positive length. Let $\mathcal F$ be the family of all distinct such faces and
write $L(F)$ for the length of $F\in\mathcal F$. The relative interiors of
distinct faces in $\mathcal F$ are disjoint. Indeed, if their supporting lines
are distinct, an intersection of the relative interiors would contradict the
fact that the whole body lies on one side of each supporting line. If the lines
coincide, then their intersections with $K$, namely the support faces
themselves, coincide.

For each $n\in\mathbb N$ we set
\[
 \mathcal F_n\eq
 \{F\in\mathcal F:L(F)\geq 1/n\}.
\]
Since the relative interiors of these segments are pairwise disjoint subsets
of the rectifiable curve $\partial K$, every finite subfamily
$\mathcal G\subset\mathcal F_n$ satisfies
\[
 \frac{\#\mathcal G}{n}
 \leq \sum_{F\in\mathcal G}L(F)
 \leq P(K).
\]
Therefore $\mathcal F_n$ is finite. Since every segment of positive length
belongs to some $\mathcal F_n$, the family
$\mathcal F=\bigcup_{n\geq1}\mathcal F_n$ is at most countable.

Finally, every maximal nondegenerate support face has a unique outer unit
normal. Hence the directions $u$ for which $\mathcal F_K(u)$ is not a
singleton also form an at most countable set.
\end{proof}

\subsection{Projection densities of area and boundary length}
\label{app:projection-densities}

Retain the coordinates from the proof of
Proposition~\ref{prop:projection-reduction} and let
\[
 \pi(\ell t,y)\eq t
\]
be the normalized projection coordinate.

\subsubsection{Projection of a uniform measure of area}
\label{app:area-projection-density}

Let $I$ be uniformly distributed in $K$. For any Borel set
$A\subset[0,1]$, Fubini's theorem and the substitution $x=\ell t$ give
\[
\begin{aligned}
 \mathbb P\{\pi(I)\in A\}
 &=\frac1{|K|}
   \iint_K\mathbf 1_{\{x/\ell\in A\}}\,\dd x\,\dd y\\
 &=\frac1{|K|}
   \int_A\int_{v(t)}^{u(t)}\ell\,\dd y\,\dd t\\
 &=\int_A\frac{\ell\bigl(u(t)-v(t)\bigr)}{|K|}\,\dd t
  =\int_A\frac{\ell h(t)}{|K|}\,\dd t.
\end{aligned}
\]
Thus the projection density is
\[
 \rho_{\mathrm{in}}(t)=\frac{\ell h(t)}{|K|}.
\]
Its normalization also follows from the section formula
\[
 |K|=\ell\int_0^1h(t)\,\dd t.
\]

\subsubsection{Projection of normalized boundary length}
\label{app:boundary-projection-density}

Let $B$ be uniformly distributed on $\partial K$ with respect to arc length.
In the nonexceptional direction under consideration, the boundary, apart from
the two common extreme points, consists of the graphs
\[
 r_u(t)=(\ell t,u(t)),\qquad r_v(t)=(\ell t,v(t)).
\]
The standard arc-length formula for graphs of convex and concave functions
gives, almost everywhere,
\[
 \dd s_u=\lVert r_u'(t)\rVert\,\dd t
   =\sqrt{\ell^2+u'(t)^2}\,\dd t,\qquad
 \dd s_v=\sqrt{\ell^2+v'(t)^2}\,\dd t.
\]
Therefore, for any Borel $A\subset[0,1]$
\[
\begin{aligned}
 \mathbb P\{\pi(B)\in A\}
 &=\frac1{P(K)}
   \left(
    \int_A\sqrt{\ell^2+u'(t)^2}\,\dd t
    +\int_A\sqrt{\ell^2+v'(t)^2}\,\dd t
   \right)\\
 &=\int_A\frac{q(t)}{P(K)}\,\dd t.
\end{aligned}
\]
Thus,
\[
 \rho_{\mathrm{bd}}(t)=\frac{q(t)}{P(K)}.
\]
In particular,
\[
 P(K)=\int_0^1q(t)\,\dd t,
\]
which confirms the normalization. If one of the extreme support faces were a
nondegenerate vertical segment, its length would produce an atom at $t=0$ or
$t=1$. This is why such exceptional directions were excluded at the beginning
of the proof of Proposition~\ref{prop:projection-reduction}.

\subsection{Classical formula for the Gini mean difference}
\label{app:gini-cdf-proof}

For completeness, we give a short standard proof of
formula~\eqref{eq:gini-cdf}. It is not new; this is the classical layer-cake
argument written in terms of indicator functions and Tonelli's theorem, and is
equivalent to the derivation in
\cite[Section~2.1.2]{YitzhakiSchechtman2013}.

Let $X$ and $Y$ be independent and have distribution $\sigma$. For any
$v,w\in[0,1]$
\[
 |v-w|=\int_0^1
 \left|\mathds{1}_{\{v\leq x\}}-\mathds{1}_{\{w\leq x\}}\right|\dd x,
\]
since the integrand equals one precisely between $v$ and $w$. The integrand is
nonnegative, so Tonelli's theorem gives
\begin{align*}
 D(\sigma)
 &=\E|X-Y|\\
 &=\int_0^1\E
 \left|\mathds{1}_{\{X\leq x\}}-\mathds{1}_{\{Y\leq x\}}\right|\dd x\\
 &=\int_0^1\bigl(\mathbb P\{X\leq x<Y\}
                  +\mathbb P\{Y\leq x<X\}\bigr)\dd x\\
 &=2\int_0^1S(x)(1-S(x))\dd x.
\end{align*}
The last equality uses independence of $X,Y$ and the identity
$\mathbb P\{X\leq x<Y\}=S(x)(1-S(x))$. Finally,
\[
 S(x)(1-S(x))=\frac14-\left(S(x)-\frac12\right)^2,
\]
which gives the second equality in~\eqref{eq:gini-cdf}.

\subsection{Reproducible materials}

The source text, the complete list of $6492$ rational Bernstein coefficients,
two implementations of the exact verification, and the build logs are included
in the reproducibility package accompanying the paper. The SHA-256 checksum of
the combined report is
\begin{center}
\footnotesize\texttt{0e81ed97d1d9d69b840b436923e34f5821fbf49fda9199e8ae77e173d1e1ddbe}.
\end{center}

\section*{Author's note added to the English translation
(August 2026)}

The original Russian-language manuscript, \emph{The Zaporozhets--Tarasov
Inequality for an Arbitrary Planar Convex Body}, was completed on 26 July 2026
and was deposited in a private GitHub repository on the same date. The exact
state of the manuscript deposited at that time is preserved in commit
\texttt{17c5b853b52146f7960f286565a1492fb334424c}%
\footnote{\url{https://github.com/mkukushkin2004/mean-distance-inequality-proof/commit/17c5b853b52146f7960f286565a1492fb334424c}}.
The repository was made public on 13 August 2026.

The present English text is a faithful translation of the Russian manuscript.
No mathematical statement or proof from the original version has been altered.

A complementary higher-dimensional result was obtained by Alexey Lotnikov in
his Bachelor's thesis, \emph{Mean distance between points inside and on the
boundary of a convex body}, defended at Saint Petersburg State University on
15 June 2026. That thesis contains counterexamples to
the conjecture in every dimension $d\geq3$. The exact Russian version of the
thesis as defended has not yet been placed in a stable public archive.

During the preparation of the present English translation, the author became
aware of the independent preprint by Eric Shen, \emph{The
Zaporozhets--Tarasov Conjecture on Mean Distances}, arXiv:2608.06470v1, first
submitted on 6 August 2026. Shen's preprint contains both a
proof of the conjecture in dimension two and counterexamples in every dimension
$d\geq3$. The proof of the planar case given in the present manuscript was
obtained independently of Shen's work.

The present section, including the information about Lotnikov's original
Bachelor's thesis and Shen's preprint, was added in August 2026 and was not
part of the Russian manuscript completed in July 2026.

\clearpage
\printbibliography[title={References}]

@article{Khinchin1938,
  author       = {Khinchin, Aleksandr Ya.},
  title        = {On Unimodal Distributions},
  journaltitle = {Izvestiya Nauchno-Issledovatel'skogo Instituta Matematiki i Mekhaniki pri Tomskom Gosudarstvennom Universitete im. V. V. Kuibysheva},
  year         = {1938},
  volume       = {2},
  number       = {2},
  pages        = {1--7},
  note         = {In Russian},
  langid       = {english},
}

@book{KolmogorovFomin1972,
  author    = {Kolmogorov, Andrey N. and Fomin, Sergei V.},
  title     = {Elements of the Theory of Functions and Functional Analysis},
  edition   = {3},
  location  = {Moscow},
  publisher = {Nauka},
  year      = {1972},
  pagetotal = {496},
  note      = {In Russian},
  langid    = {english},
}

@book{YitzhakiSchechtman2013,
  author    = {Yitzhaki, Shlomo and Schechtman, Edna},
  title     = {The Gini Methodology: A Primer on a Statistical Methodology},
  series    = {Springer Series in Statistics},
  location  = {New York},
  publisher = {Springer},
  year      = {2013},
  langid    = {english},
}

@book{Gini1912,
  author    = {Gini, Corrado},
  title     = {Variabilità e mutabilità: contributo allo studio delle
               distribuzioni e delle relazioni statistiche},
  series    = {Studi economico-giuridici pubblicati per cura della Facoltà
               di Giurisprudenza della R. Università di Cagliari},
  number    = {Anno III, parte II},
  location  = {Bologna},
  publisher = {Tipografia di Paolo Cuppini},
  year      = {1912},
  langid    = {italian},
}

@article{Gini1914,
  author       = {Gini, Corrado},
  title        = {Sulla misura della concentrazione e della variabilità dei
                  caratteri},
  journaltitle = {Atti del Reale Istituto Veneto di Scienze, Lettere ed Arti},
  year         = {1914},
  volume       = {73},
  number       = {2},
  pages        = {1203--1248},
  langid       = {italian},
}

@article{BonnetEtAl2021,
  author       = {Bonnet, Gilles and Gusakova, Anna and Thäle, Christoph and Zaporozhets, Dmitry},
  title        = {Sharp inequalities for the mean distance of random points in convex bodies},
  journaltitle = {Advances in Mathematics},
  year         = {2021},
  volume       = {386},
  eid          = {107813},
  langid       = {english},
}

@article{Tokmachev2022,
  author       = {Tokmachev, Alexander S.},
  title        = {The Mean Distance between Random Points on the Boundary of a Convex Figure},
  journaltitle = {Zapiski Nauchnykh Seminarov POMI},
  year         = {2022},
  volume       = {510},
  pages        = {248--261},
  note         = {In Russian},
  langid       = {english},
}

@article{Lotnikov2025,
  author       = {Lotnikov, Alexey S.},
  title        = {The Mean Distance between Random Points in Centrally Symmetric Convex Bodies},
  journaltitle = {Zapiski Nauchnykh Seminarov POMI},
  year         = {2025},
  volume       = {544},
  pages        = {211--235},
  note         = {In Russian},
  langid       = {english},
}

@unpublished{Tokmachev2026,
  author = {Tokmachev, Alexander S.},
  title  = {Inequalities for convex functions of random points inside and on the boundary of convex bodies},
  year   = {2026},
  note   = {arXiv preprint 2607.08869},
  langid = {english},
}

@article{MichelucciFoufouKubicki2012,
  author       = {Michelucci, Dominique and Foufou, Sebti and Kubicki, Arnaud},
  title        = {On the Complexity of the Bernstein Combinatorial Problem},
  journaltitle = {Reliable Computing},
  year         = {2012},
  volume       = {17},
  number       = {1},
  pages        = {22--33},
  langid       = {english},
}

@book{Rockafellar1973,
  author     = {Rockafellar, R. Tyrrell},
  title      = {Convex Analysis},
  translator = {Ioffe, A. D. and Tikhomirov, V. M.},
  location   = {Moscow},
  publisher  = {Mir},
  year       = {1973},
  pagetotal  = {469},
  note       = {Russian translation},
  langid     = {english},
}

@article{EwaldLarmanRogers1970,
  author       = {Ewald, G. and Larman, D. G. and Rogers, C. A.},
  title        = {The Directions of the Line Segments and of the
                  {$r$}-Dimensional Balls on the Boundary of a Convex Body
                  in Euclidean Space},
  journaltitle = {Mathematika},
  year         = {1970},
  volume       = {17},
  number       = {1},
  pages        = {1--20},
  langid       = {english},
}

\end{document}